\documentclass[3p]{elsarticle}
\usepackage{amscd, amssymb, amsthm}
\usepackage[all]{xy}
\usepackage{color}
\usepackage{url}
\usepackage{amscd, amssymb, amsfonts, amsthm, amsmath}
\usepackage{mathdots}
\usepackage[breaklinks=true]{hyperref}
\usepackage{multicol}
\usepackage{yfonts}

\usepackage[normalem]{ulem}

\usepackage[labelsep=endash]{caption}

\numberwithin{equation}{section}
\usepackage[labelsep=endash]{caption}

\makeatletter
\def\ps@pprintTitle{%
  \let\@oddhead\@empty
  \let\@evenhead\@empty
  \let\@oddfoot\@empty
  \let\@evenfoot\@oddfoot
}
\makeatother

\newcommand{\soc}{\mbox{\rm{Soc}}}
\newcommand{\topp}{\mbox{\rm{Top}}}
\newcommand{\length}{\mbox{\rm{l}}}
\newcommand{\Llength}{\mbox{\rm{ll}}}
\newcommand{\start}{\mbox{\rm{s}}}
\newcommand{\target}{\mbox{\rm{t}}}
\newcommand{\gl}{\mbox{\rm{gldim}}}
\newcommand{\fin}{\mbox{\rm{findim}}}
\newcommand{\prj}{\mbox{\rm{P}}}
\newcommand{\fidim}{\mbox{\rm{$\phi$dim}}}
\newcommand{\psidim}{\mbox{\rm{$\psi$dim}}}
\newcommand{\findim}{\mbox{\rm{findim}}}
\newcommand{\repdim}{\mbox{\rm{repdim}}}
\DeclareMathOperator{\Ob}{Ob}
\def\ker{\mbox{\rm{Ker}}}
\def\I{\mbox{\rm{Im}}}
\def\mod{\mbox{\rm{mod}}}
\def\ind{\mbox{\rm{ind}}}
\def\add{\mbox{\rm{add}}}

\def\pd{\mbox{\rm{pd}}}
\def\id{\mbox{\rm{id}}}
\def\mini{\mbox{\rm{min}}}
\def\ext{\mbox{\rm{Ext}}}

\def\hom{\mbox{\rm{Hom}}}
\def\enn{\hbox{\rm{End}}}

\def\rad{\hbox{\rm{rad}}}
\def\Obj{\mbox{\rm{Obj}}}
\def\fpd{\mbox{\rm{fpd}}}
\def\fid{\mbox{\rm{fid}}}
\def\Rep{\mbox{\rm{Rep}}}
\def\Fact{\mbox{\rm{Fact}}}
\def\st{\mbox{\rm{s}}}
\def\tg{\mbox{\rm{t}}}

\begin{document}
\newcommand{\mono}[1]{%
\gdef\puA{#1}}
\newcommand{\puA}{}
\newcommand{\faculty}[1]{%
\gdef\puC{#1}}
\newcommand{\puC}{}
\newcommand{\facultad}[1]{%
\gdef\puD{#1}}
\newcommand{\puD}{}
\newcommand{\N}{\mathbb{N}}
\newcommand{\Z}{\mathbb{Z}}
\newtheorem{teo}{Theorem}[section]
\newtheorem{prop}[teo]{Proposition}
\newtheorem{lema}[teo] {Lemma}
\newdefinition{ej}[teo]{Example}
\newtheorem{obs}[teo]{Remark}
\newtheorem{defi}[teo]{Definition}
\newtheorem{coro}[teo]{Corollary}
\newtheorem{nota}[teo]{Notation}



\title{A survey on Igusa-Todorov functions}

\author[add]{Marcos Barrios}
\ead{marcosb@fing.edu.uy} 

\author[add]{Marcelo Lanzilotta}
\ead{marclan@fing.edu.uy}

\author[add]{Gustavo Mata\corref{cor}}
\ead{gmata@fing.edu.uy}
\cortext[cor]{Corresponding Author}

\address[add]{Universidad de La Rep\'ublica, Facultad de Ingenier\'ia -  Av. Julio Herrera y Reissig 565, Montevideo, Uruguay}

\begin{abstract}
In this survey, we review the fundamental properties of the Igusa-Todorov functions, the $\phi$-dimension, the $\psi$-dimension and their generalizations. 

\end{abstract}


\maketitle

\section{Introduction}
In an attempt to prove the finitistic dimension conjecture, Igusa and Todorov defined in \cite{IT} two homological functions from the objects of $\mod A$ (the category of right finitely generated modules over an Artin algebra $A$) to the natural numbers, which generalizes the notion of projective dimension. Using these functions, they showed that the finitistic dimension of Artin algebras with representation dimension at most three is finite. Nowadays, these functions are known as the Igusa-Todorov functions, usually denoted by $\phi$ and $\psi$.
Several recent works dedicated to the studying of these functions and their associated dimensions show the growing relevance of the subject.
For example, the functions were used as a tool to give a proof of the finitistic dimension conjecture for some 
families of algebras (see e.g. \cite{HLM09, IT, We, WX, Xi, Xu}).
In \cite{HLM09}, the authors show that given an Artin algebra $A$, such that its module category has at most three radical layers of infinite projective dimension, then $A$ has finite finitistic dimension.
In \cite{We}, the author proves that Igusa-Todorov algebras verifies the finitistic dimension conjecture using the $\psi$ function. In \cite{WX} and \cite{Xi}, the authors show that under specific hypotheses, for a given finite finitistic dimensional algebras, some families of subalgebras have finite finitistic dimension.
Finally, in \cite{Xu}, the author introduces a new function from the bounded derived category of a finite dimensional algebra over a field to the set of natural numbers, which is a generalized version of the Igusa-Todorov function. As an application, he gives a new proof of the finiteness of the finitistic dimension of special biserial algebras.

We also believe that the study of the Igusa-Todorov functions is of intrinsic interest, motivated by diverse articles on the subject (see e.g. \cite{BMR, FLM, HL, LM, LM2017, LMM, M, Q}).

\section{Preliminaries}

Throughout this article, we work with connected Artin algebras. Let $A$ be a connected Artin R-algebra (for instance, a finite dimensional basic algebra defined over a field $\Bbbk$). The category of finite dimensional right $A$-modules will be denoted by $\mod A$, the indecomposable modules of $A$ by $\ind A$, and the set of isoclasses of simple $A$-modules by $\mathcal{S} (A)$ (for short $\mathcal{S}$).

For an $A$-module $M$, $\add M \subset \mod A$ is the full subcategory formed by direct sums of the direct summands of $M$, and we denote by $\soc(M)$ its socle, by $\topp(M)$ its top, and by $\Llength(M)$ its Loewy length. In particular, we use the notation $A_0 = \topp(A)$. We also denote by $P(M)$ and $I(M)$ the projective cover and the injective envelope of the $A$-module $M$, respectively. $\mathcal{P}_A$ ($\mathcal{I}_A$) is the full subcategory of projective (injective) modules of $\mod A$. We denote by $\mathcal{S}_P$, by $ \mathcal{S}_I$ the projective simple modules and the injective simple modules, respectively and by $\mathcal{S}_D = \mathcal{S}\setminus (\mathcal{S}_P \cup \mathcal{S}_I)$.

Given $M$ in $\mod A$, we denote its projective dimension by $\pd_A(M)$, its injective dimension by $\id_A(M)$, and the $n^{th}$-syzygy by $\Omega_A^n(M)$. In case there is no possible misinterpretation, we denote them by $\pd (M)$, $\id (M)$, and $\Omega^n(M)$, respectively. We recall that the global dimension (finitistic dimension) of a subcategory $\mathcal{X} \subset \mod A$, which we denote by $\gl(\mathcal{X})$ ($\fin(\mathcal{X})$), is the supremum  of the set of projective dimensions of $A$-modules in $\mathcal{X}$ (of finite projective dimension). The global dimension can be a natural number or infinite. In case $\mathcal{X} = \mod A$ we use the notations $\gl (A)$ and $\fin (A)$ instead of $\gl (\mod A)$ and $\fin (\mod A)$.  We recall that $\repdim(A) \leq n$ if there is a f.g. module $M$ such that $\gl(\enn_A(M)^{op}) \leq n$ and $\add (M)$ contains all projective and all injective $A$-modules.

Let $M$ and $N$ be in $\mod A$, an homomorphism $M \stackrel{f}{\rightarrow} N$ of $A$-modules is called a radical map if $gfh$ is not an isomorphism for any $N \stackrel{g}{\rightarrow} X$ and $X \stackrel{h}{\rightarrow} M$ with $X$ an indecomposable A-module. The set $\rad_A(M,N)$ is the subset of $\hom_A(M,N)$ formed by all radical maps. 

We denote by $\mathcal{C}_A$ the abelian category of all $R$-functors $F : \mod A \rightarrow \mod R $. Let $\ext_A^n(M,\ \cdot \ )$ be the $n^{th}$-left derived functor of $\hom_A (M,\ \cdot \ )$. We denote by $ ^\bot A$ the full subcategory of $\mod A$, also called stable modules (see \cite{AB}), whose objects are those $M \in \mod A$ such that $\ext_A^i(M, A) = 0$ for $i\geq 1$ and by $\mathcal{G} P_A$ the full subcategory of $\mod A$ whose objects are the Gorenstein projective modules. It is well known that $\mathcal{G} P_A$ is a subcategory of $ ^\bot A$, see \cite{Zh}.

Let $\mathcal{C}$ be an additive category, a complex $P^{\bullet}$ over $\mathcal{C}$ is a sequence of morphisms $d_{P^{\bullet}}^i$  between objects $P^i$ in $\mathcal{C}$ : $\xymatrix{ \cdots \ar[r]^{d_{P^{\bullet}}^{i-1}} & P^{i-1} \ar[r]^{d_{P^{\bullet}}^i} & P^{i}\ar[r]^{d_{P^{\bullet}}^{i+1}} & P^{i+1} \ar[r] & \cdots}$ such that $d_{P^{\bullet}}^{i} d_{P^{\bullet}}^{i-1} = 0$. For a complex $P^{\bullet}$, $[1]$ denotes the shift functor, that is, $P^{\bullet} [1]^j = P^{j+1}$ and $d_{P^{\bullet} [1]}^j = -d_{P^{\bullet}}^{j+1}$ for any $j \in \mathbb{Z}$, and $\tau_{\leq i} (P^{\bullet})$ denotes a new complex defined by $\tau_{\leq i} (P^{\bullet} )^j = P^j$ for all $j \leq i$, and zero otherwise. The category of complexes over $\mathcal{C}$ with chain maps is denoted by $C(\mathcal{C})$. The homotopy category of complexes over $\mathcal{C}$ is denoted by $\mathcal{K}(\mathcal{C})$. The full subcategory of $\mathcal{K} (\mathcal{C})$ consisting of bounded above (resp. bounded) complexes over $\mathcal{C}$ is denoted by $\mathcal{K}^{-} (\mathcal{C})$ (resp. $\mathcal{K}^b (\mathcal{C})$). The full subcategory of $\mathcal{K}^{-} (\mathcal{C})$ consisting of complexes over $\mathcal{C}$ with only finitely many nonzero cohomologies is denoted by $\mathcal{K}^{-,b} (\mathcal{C})$. For a finite dimensional algebra $A$, $\mathcal{D}(\mod A)$ and $\mathcal{D}^{b}(\mod A)$ are the derived category and the derived category of bounded complexes over $\mod A$ respectively. An object $X$ in $\mathcal{D}(\mod A)$ is said to be compact if $Hom_{\mathcal{D}(\mod A)} (X,\ \cdot)$ commutes with direct sums. It is well known that up to isomorphism the objects in $\mathcal{K}^b(\mathcal{P}_A)$ are precisely all the compact objects in $\mathcal{D}(\mod A)$ (see \cite{BN}).

Consider $\mathcal{C}$ an abelian category and $X^{\bullet}$ a complex over $\mathcal{C}$. The $i$-th cohomology group is
denoted by $H^i (X^{\bullet})$. A morphism $f : X^{\bullet} \rightarrow Y^{\bullet}$ of $C (\mathcal{C})$ is a quasi-isomorphism if the induced morphisms $H^i (f) : H^i (X^{\bullet}) \rightarrow H^i (Y^{\bullet})$ are isomorphisms for
all $i$, and we say that $X^{\bullet}$ is quasi-isomorphic to $Y^{\bullet}$ in this case.

A complex $P^{\bullet}$ over $\mod A$ is called a radical complex if all of its differential maps are radical maps. For convenience, we write $\mathcal{RK}^{\leq 0}(\mathcal{P}_A) = \{ P^{\bullet} \in \mathcal{K}^{-,b}(\mathcal{P}_A) \text{: } P^{\bullet} \text{ is a radical complex with } P^j = 0, \text{ for each } j \geq 1\}$.

\begin{obs} \label{Remark1}
Two radical complexes are isomorphic in $\mathcal{K} (\mod A)$ if and
only if they are isomorphic in $C (\mod A)$.
\end{obs}

Thanks to Remark \ref{Remark1}, we can define the projective dimension for an object $P^{\bullet}$ in $\mathcal{K}^{-} (\mathcal{P}_A)$, $ \pd (P^{\bullet}) = \sup \{n \text{: } Q^{-n} \not =0\text{, } Q^{\bullet} \text{ is a radical complex in } \mathcal{K}^{-} (\mathcal{P}_A) \text{ and } Q^{\bullet} \simeq P^{\bullet} \text{ in } \mathcal{K}^{-} (\mathcal{P}_A) \}$ and $\pd (P^{\bullet}) = 0$ if $P^{\bullet} = 0$ in $\mathcal{K}^{-} (\mathcal{P}_A)$.

\begin{obs}\label{Remark2}
Let $X^{\bullet}$ be a bounded above complex over $\mod A$. It is known that there exists a radical complex $P_{X^{\bullet}}^{\bullet} \in \mathcal{K}^{-} (\mathcal{P}_A)$ such that $P_{X^{\bullet}}^{\bullet} $ and $X^{\bullet}$
are quasi-isomorphic.
\end{obs}

We can define $\pd (X^{\bullet}) = \pd (P_{X^{\bullet}}^{\bullet})$ and $\sup (X^{\bullet}) = -\max\{i: P^{i}_{X^{\bullet}} \not = 0  \}$ by Remark \ref{Remark2}. Finally, the homological width of $X^{\bullet}$ is defined as $w(X^{\bullet}) = -\sup(X^{\bullet} ) + \pd(X^{\bullet})$.\\

The proofs of several results through this article use the following version of the Fitting's lemma.
\begin{lema}\label{Fitting}(Fitting's Lemma) Let $R$ be a noetherian ring. Consider a left $R$-module $M$
and $f \in End_R (M)$. Then, for any finitely generated $R$-submodule $X$ of $M$, there is a non-negative integer
$$\eta_f(X) = \min\{k \in N : f |_{f^m (X)} : f^m (X) \rightarrow f^{m+1} (X),\text{ are isomorphic } \forall m \geq k\}.$$

Furthermore, for any $R$-submodule $Y$ of $X$, we have that $\eta_f (Y) \leq \eta_f (X)$.

If $R$ is an Artin algebra and $X = M$ there is a direct sum decomposition $X = Y \oplus Z$ so that $Z = \ker f^m$ and $Y = \I f^m$  for all $m \geq \eta_f (X)$.
\end{lema}

If $Q = (Q_0, Q_1, \start, \target)$ is a finite connected quiver, $\Bbbk Q$ denotes its path algebra. Given $\rho$ a path in $\Bbbk Q$, $\length(\rho)$, $\start(\rho)$ and $\target(\rho )$ denote the length, start and target of $\rho$ respectively, as well $\mathbb{P}^k$ ($\mathbb{P}^{\geq k}$) denotes the set of paths of length equal (greater or equal) to $k$, for $k \in \mathbb{Z}^{+}$. We say that a quiver $Q$ is strongly connected if, for every pair of vertices $v, w \in Q_0$, there is a path from $v$ to $w$. $J_Q$ denotes the ideal of $\Bbbk Q$ generated by the arrows of $Q$. In case there is no possible misinterpretation, we denote it by $J$. $C^m$ denotes the oriented cycle graph with $m$ vertices.

\begin{defi}
Given $A$ a finite dimensional basic algebra, we denote the rule $\nu$ by
\begin{align*}
\nu:\mathcal{S}(A)&\rightarrow \add(\topp(A))\\
S&\mapsto\soc(\prj(S)).
\end{align*}
\end{defi}
The following theorem due to Nakayama will be useful.
\begin{teo}(\cite{N})\label{Nakayama}
Let $A$ be a finite dimensional basic algebra. The following statements are equivalent.
\begin{enumerate}
\item  $A$ is selfinjective.
\item The rule $S \rightarrow \soc(\prj(S))$, defines a permutation $\nu :\mathcal{S} (A) \rightarrow \mathcal{S}(A).$
\end{enumerate}  
\end{teo}

In case $A$ is a selfinjective algebra, we call $\nu$ the {\bf Nakayama Permutation of $A$}.

\section{Igusa-Todorov functions} 
\subsection{First concepts}
We recall the definition of $\phi$ and $\psi$, see \cite{IT}, the Igusa-Todorov functions (from now on IT-functions), and some basic properties of them. We also remind the definition of the $\phi$-dimension and $\psi$-dimension for any subcategory of $\mod A$.

\begin{defi}(\cite{IT})
Let  $K_0(A)$ be the abelian group generated by all symbols $[M]$, with $M \in \mod A$, modulo the relations
\begin{enumerate}
  \item $[M]-[M']-[M'']$ if  $M \cong M' \oplus M''$,
  \item $[P]$ for each projective module $P$.
\end{enumerate}
\end{defi}

Let $\bar{\Omega}: K_0 (A) \rightarrow K_0 (A)$ be the group endomorphism induced by $\Omega$, and let $K_i (A) = \bar{\Omega}(K_{i-1}(A))= \ldots = \bar{\Omega}^{i}(K_{0} (A))$. If $M$ is a finitely generated $A$-module, then $\langle \add M \rangle$ denotes the subgroup of $K_0 (A)$ generated by the classes of indecomposable summands of $M$.
Given a finitely generated subgroup $G$ of $K_0(A)$, we define $\eta_{\bar{\Omega}}(G)$ as in Lemma \ref{Fitting} (Fitting's Lemma). 


\begin{defi}(\cite{IT}) \label{monomorfismo}
The (right) Igusa-Todorov function $\phi_A$ of $M\in \mod A$  is defined as:

\[\phi_{A}(M) := \eta_{\bar{\Omega}}(\langle \add M \rangle)=  \min\left\{l: \bar{\Omega}_{A} {\vert}_{{\bar{\Omega}_{A}}^{l+s}(\langle \add M \rangle)} \text{ is a monomorphism, for all }s \geq 0 \right\} .\]
\end{defi}

We define dually the left Igusa-Todorov $\phi$ function for an Artin algebra $A$ as the right Igusa-Todorov function $\phi_{A^{op}}$. In case there is no possible misinterpretation, we denote by $\phi$ the Igusa-Todorov function $\phi_A$.\\

As we already observed, Igusa and Todorov defined in \cite{IT} this  homological function from the objects of $\mod A$ to the natural numbers, which generalizes the notion of projective dimension. The first item in the following proposition shows this property.

\begin{prop}(Lemma 3.4 of \cite{HLM} and Lemma 1.2 \cite{IT}) \label{it1} \label{Huard1}
If $M,N\in\mod A$, then we have the following.

\begin{enumerate}
  \item $\phi(M) = \pd (M)$ if $\pd (M) < \infty$.
  \item $\phi(M) = 0$ if $M \in \ind A$ and $\pd(M) = \infty$.
  \item $\phi(M) \leq \phi(M \oplus N)$.
  \item $\phi\left(M^{k}\right) = \phi(M)$ for $k \geq 1$.
  \item $\phi(M) \leq \phi(\Omega(M))+1$.
\end{enumerate}
\begin{proof}
For statements 1-4 see \cite{IT}, and for 5 see \cite{HLM}.
\end{proof}
\end{prop}

Now, we introduce the definition of the second Igusa-Todorov function $\psi$. There are two main purposes to define this second notion. First, to assemble the measure of finitistic dimensions of subcategories with the first IT-function $\phi$. Second, just the next function serves as a bound in the main Theorem \ref{cotadp} of \cite{IT}.

\begin{defi}(\cite{IT}) \label{ITpsi}
The (right) Igusa-Todorov function
$\psi_{A}$ of $M\in \mod A$  is defined as 
\begin{equation*}
\psi_{A}(M) = \phi_{A}(M) + \sup\{ \pd(N) : \Omega^{\phi_A(M)}(M) = N
\oplus N'\mbox{ and } \pd(N)< \infty\}.
\end{equation*}
\end{defi}

 We define dually the left Igusa-Todorov $\psi$ function for an Artin algebra $A$ as the right Igusa-Todorov function $\psi_{A^{op}}$. In case there is no possible misinterpretation, we denote by $\psi$ the Igusa-Todorov function $\psi_A$.
This second function, which also generalizes the projective dimension, has good properties and is an interesting homological measure.

\begin{prop}(Proposition 3.5 of \cite{HLM} and Lemma 1.3 of \cite{IT})\label{Huard2} 
If $M,N\in\mod A$, then we have the following.
\begin{enumerate}
\item $\psi(M) = \pd(M)$ if $\pd (M) < \infty$.
\item $\psi(M) = 0$ if $M \in \ind A$ and $\pd(M) = \infty$.
\item $\psi(M) \leq\psi(M \oplus N)$.
\item $\psi(M^{k}) = \psi(M)$ for $k \geq 1$.
\item If $N$ is a direct summand of $\Omega^{n}(M)$ where $n\leq\phi(M)$ and
$\pd(N) < \infty$, then $\pd(N) + n \leq\psi(M)$.
\item $\psi(M) \leq\psi(\Omega(M))+1$.
\end{enumerate}
\end{prop}
\begin{proof}
For statements 1--5 see \cite{IT}, and for 6 see
\cite{HLM}.
\end{proof}

\begin{defi}(\cite{HL})
Let $A$ be an Artin algebra, and $\mathcal{C}$ a full subcategory of
$\mathrm{mod}A$. We define
$$\fidim(\mathcal{C}) = \sup\{\phi(M): M \in \Ob\mathcal{C} \},\text{ and }
\ \psidim(\mathcal{C}) = \sup\{\psi(M): M \in \Ob\mathcal{C} \}.$$
In particular, we denote by
$$\fidim(A) = \fidim(\mod A),\text{and } \psidim(A) = \psidim(\mod A).$$
\end{defi}

It is easy to see that the following inequalities hold for an Artin algebra $A$:

\begin{equation} \label{desigualdades}
\fin (A) \leq \fidim (A) \leq \psidim (A) \leq \gl (A) 
\end{equation}

$$ \psidim (A) \leq 2. \fidim (A).$$

In particular, the finiteness of the $\phi$-dimension is equivalent to the finiteness of the $\psi$-dimension, and this implies the finiteness of the finitistic dimension.

\begin{obs} 
\begin{itemize}
\item
Using Propositions \ref{Huard1} and \ref{Huard2}, items 3. and 4., then $\phi(M) = \fidim (\add M)$ and $\psi(M) = \psidim (\add M)$, for all $M \in\mod A$.
\item
{\label{finitesubcategory}}
Recall that given a subcategory $\mathcal{C} \subset \mod A$, $\mathcal{C}$ is of finite representation type
if there exists $N \in \mod A$, such that for all $M \in \mathcal{C}$, $M \in \add N$.
If $\mathcal{C}$ is of finite representation type then $\fidim (\mathcal{C}) < \infty$.
\end{itemize}

\end{obs}
 
The following proposition can be obtained using simple linear algebra arguments. See Proposition 3.6 of \cite{LMM}. 

\begin{prop}(Proposition 3.6 of \cite{LMM}) \label{invariante} Let $M \in \mod A$ such that $\Omega(M) \in \add M$. If $M =\displaystyle  \bigoplus_{i = 1}^n M^{\alpha_i}_i$ is the decomposition into indecomposable $A$-modules of $M$, then $\phi(M) \leq n$.
\end{prop}

\subsection{Igusa-Todorov's Theorem}

Given a short exact sequence in $\mod A$:  $$\xymatrix{0 \ar[r] & M' \ar[r] & M \ar[r] & M'' \ar[r] & 0}$$
it is well known that $\pd (M'') \leq \pd(M'\oplus M) + 1$. In particular, the inequality is helpful if two of (and then) the three modules have finite projective dimensions. The following theorem generalizes the above inequality using the second IT-function. It is the main Theorem of \cite{IT}. It needs only one module with finite projective dimension.

\begin{teo} \label{cotadp}(Theorem 1.4 of \cite{IT})
Suppose that $\xymatrix{0 \ar[r] & M' \ar[r] & M \ar[r] & M'' \ar[r] & 0}$ is a short exact sequence of
f.g. $A$-modules and $M''$ has finite projective dimension. Then $\pd (M'') \leq \psi(M'\oplus M) + 1$.
\end{teo}

We offer the proof for the convenience of the reader. The proof given here (see \cite{M*}) is slightly different from the one provided in \cite{IT}.   

\begin{proof} Let $r = \pd (M'')$. Since $\pd (M'')<\infty$ then $\bar{\Omega}^l([M']) = \bar{\Omega}^{l}([M])$ for some $l \geq 0$. Let $n = \mini \{l : \bar{\Omega}^l([M']) = \bar{\Omega}^{l}([M]) \}$, then $n \leq \pd (M'')$ and since $[M],[M'] \in \langle \add (M\oplus M')\rangle$ we have $n \leq \phi(M \oplus M')$. The $n$-th syzygy of our short exact
sequence gives an exact sequence of the form:
\begin{equation}\label{ecuacion}
\xymatrix{0 \ar[r]& X\oplus P \ar[r]^{t}& X \oplus Q \ar[r]^{g}& \Omega^{n}(M'') \ar[r]& 0}
\end{equation}

where $P$ and $Q$ are projectives and the map $t$ is given by the following matrix
$$t = \left(
       \begin{array}{cc}
         f & h_1 \\
         h_2 & h_3 \\
       \end{array}
     \right).$$

Since $f \in \enn (X)$, by Lemma \ref{Fitting}, we have the decomposition $X = Y \oplus Z$ and the map $f$ can be represented by
$$f = \left(
       \begin{array}{cc}
         f_{11} & 0 \\
         0 & \alpha \\
       \end{array}
     \right)$$
where $\alpha: Z \rightarrow Z$ is nilpotent and $f_{11}: Y \rightarrow Y$ is an isomorphism.

We now apply the functor $\hom(\ \cdot , M)$ to the short exact sequence (\ref{ecuacion}) and we obtain the following long exact sequence:
$$\xymatrix{\ext^{k}_A(X,M) \hspace{-3pt} \ar[r]^{\gamma_{k}}& \ext^{k}_A(X,M) \hspace{-3pt} \ar[r]^{\hspace{-20pt} \sigma_k}& \ext^{k+1}_A(\Omega^{n}(M''),M)\hspace{-3pt} \ar[r]^{\hspace{10pt}\lambda_{k+1}}& \ext^{k+1}_A(X,M) \hspace{-3pt} \ar[r]^{\gamma_{k+1}}& \ext^{k+1}_A(X,M)}\hspace{-4pt}, \ k \geq 1.$$


\noindent where $\gamma_{k} = \left(
                                \begin{array}{cc}
                                  \delta_{k} & 0 \\
                                  0 & \beta_{k} \\
                                \end{array}
                              \right)$, $\beta_k = \ext^k_A(\alpha, M)$, $\delta_k = \ext^k_A(f_{11}, M)$ and $\lambda_{k+1} = \ext^{k+1}_A(g,M)$. Note that $\beta_k$ is nilpotent for all $k \geq 1$, because $\alpha$ is nilpotent.\\

\underline{Claim:} $\pd (Z)< \infty$.

Since $\ext^{j}_A(\Omega^{n}(M''),M) = 0$ for all $j > r-n$ we have that $\gamma_k$ is an epimorphism for $k > r-n-1$. Hence $\beta_{k}$ is also an epimorphism for $k > r-n-1$ and since $\beta_{k}$ is nilpotent, then $\ext^{k}(Z,M) = 0$ for $k > r-n-1$ for all $M \in \mod A$, and follows the claim.\\

\underline{Claim:} $\pd (\Omega^n(M'') )\leq \pd (Z) + 1$.

Suppose $\ext^{k+1}_A(\Omega^{n}(M''),M) \neq 0$. Hence $\sigma_k \neq 0$ or $\lambda_{k+1} \neq 0$.

\begin{itemize}
  \item Note that $\sigma_k \neq 0$ implies $\gamma_k$, $\beta_k$ are not epimorphisms. As a consequence $\ext^{k}(Z,M) \neq 0$.
  \item Note that $\lambda_{k+1} \neq 0$ implies $\gamma_{k+1}$, $\beta_{k+1}$ are not monomorphisms. As a consequence $\ext^{k+1}_A(Z,M) \neq 0$.
\end{itemize}

We conclude that $\ext^{k+1}_A(\Omega^{n}(M''),M) \neq 0$ implies $\ext^{k}_A(Z,M) \neq 0$ or $\ext^{k+1}_A(Z,M) \neq 0$ and follows the claim.

Since $Z$ is a direct summand of $\Omega^n (M \oplus M')$, $\pd (Z) < \infty$ and $n \leq \phi(M \oplus M')$, by Proposition \ref{Huard2} we have $\pd (Z) + n \leq \psi(M \oplus M')$. Finally, by the second claim, we deduce that $\pd (M'') \leq \psi(M \oplus M') + 1$.\end{proof}

J. Cappa (UNS, Argentina) gave the following two examples. The first example shows that the $\psi$ function cannot be changed by the $\phi$ function in Theorem \ref{cotadp}, which cannot be got the upper bound using the second function. In fact the following example shows that for all $n \in \mathbb{N}$ there exists an Artin algebra $A(n)$ and a short exact sequence in $\mod A(n)$ $\xymatrix{0 \ar[r] & X \ar[r] & Y \ar[r] & M \ar[r] & 0}$ such that $n(1+\phi(X\oplus Y)) < \pd(M) < \infty$.

\begin{ej} (J. Cappa) Let $A(n) = \frac{\Bbbk Q}{I}$, where $Q$ is given by:

$$ \xymatrix{ & 2 \ar[dr]^{\alpha_2} \ar@(ru,lu)_{\gamma} & & &  & & &\\ 1 \ar[ur]^{\alpha_1} \ar[rr]^{\beta} & & 3 \ar[r]^{\alpha_3}& 4 \ar[r]^{\alpha_4} & 5 \ar[r]^{\alpha_5} & \cdots \ar[r]^{\alpha_{2n+1}} & 2n+2 }$$
and $I = \langle \gamma^2;\ \gamma \alpha_2;\ \beta \alpha_3;\ \alpha_{i} \alpha_{i+1},\ 1\leq i\leq 2n \rangle$. 

\noindent Consider the short exact sequence $\xymatrix{0 \ar[r] & S_2\ar[r] & \frac{P_1}{S_2 \oplus S_3} \ar[r] & S_1 \ar[r] & 0}$. It is easy to prove that $\pd(S_1) = 2n+1$, and on the other hand $\Omega(S_2) = S_2\oplus S_3 = \Omega(\frac{P_1}{S_2 \oplus S_3} )$, hence $\phi (S_2 \oplus \frac{P_1}{S_2 \oplus S_3}) = 1$. $\square$
\end{ej}
The following example shows that the projective dimension cannot be changed by the $\phi$ function in Theorem \ref{cotadp}, if we remove the finite projective dimension hypothesis of the third module in the exact sequence. In fact the following example shows that for all $n \in \mathbb{N}$ there exists an Artin algebra $B(n)$ and a short exact sequence in $\mod B(n)$ $\xymatrix{0 \ar[r] & X \ar[r] & Y \ar[r] & M \ar[r] & 0}$ such that $n(1+\psi(X\oplus Y)) < \phi(M) < \infty$. 

\begin{ej}(J. Cappa) Let $B(n) = \frac{\Bbbk Q}{J^2}$, where $Q$ is given by:

$$\xymatrix{ 1 \ar[r] \ar[rd] & 2 \ar[r] & \cdots \ar[r] & n+1 \ar[r] & n+2 \\ & n+4 & & & \\ n+3 \ar@(ld, rd) \ar[ru] & & & & }$$
\vspace{.5cm}

\noindent Consider the short exact sequence $0 \rightarrow P_{n+4}\rightarrow I_{n+4}\rightarrow S_1 \oplus S_{n+3} \rightarrow 0$. It is clear that $\Omega(S_1) = S_2\oplus P_{n+4}$, $\Omega^{k}(S_1) = S_{k+1}$ if $k = 2, \ldots, n$, and $\Omega^{n+1}(S_1) = S_{n+2} =  P_{n+2}$. On the other hand $\Omega(S_{n+3}) = S_{n+3} \oplus P_{n+4}$. Hence $\phi(S_1\oplus S_{n+3}) = n+2$.
The module $I_{n+4}$ is indecomposable and $\pd(I_{n+4}) = \infty$, then $\psi(I_{n+4})=\phi(I_{n+4}) = 0$.
Therefore $n = n(1+\psi(P_{n+4}\oplus I_{n+4}))< \phi(S_1\oplus S_{n+3}) = n+1 $. $\square$

\end{ej}

The next corollary follows from Theorem \ref{cotadp}. See also 
\cite{Xu}, Lemma 4.1.

\begin{coro} \label{corocotadp} (Remark 1.5 of \cite{IT})
Suppose that $\xymatrix{0 \ar[r] & M' \ar[r] & M \ar[r] & M'' \ar[r] & 0}$ is a short exact sequence of
f.g. $A$-modules, then:

\begin{enumerate}

\item if $\pd(M) < \infty$,\ $\pd M \leq \psi(  M'\oplus \Omega (M'')) + 1$,

\item if $\pd(M') < \infty$,\ $\pd M' \leq \psi( \Omega(M) \oplus \Omega (M'')) + 1$.

\end{enumerate}

\begin{proof}

\begin{enumerate}

\item
From the short exact sequence, we get the following commutative diagram with exact rows and columns:

$$\xymatrix{
 & & \ar[d]0 & \ar[d] 0 & \\
 & 0 \ar[r]\ar[d] & \Omega(M'')\ar[d]\ar[r] & \Omega(M'')\ar[r]\ar[d]^{\iota_0}&0\\ 
0\ar[r]& M' \ar[r]^{\!\!\!\!\!\!\!\!\!\!\!\!\!\!\!(1,0)^t}\ar@{=}[d] & M'\oplus P_0(M'')\ar[d]^{(f,\alpha)} \ar[r]^{\ \ \ \ (0,1)} & P_0(M'')\ar[r]\ar@{.>}[ld]^{\alpha}\ar[d]^{p_0} & 0\\ 
0\ar[r]& M' \ar[d] \ar[r]_{f}& M\ar[d] \ar[r]_{g}& M'' \ar[d]\ar[r] & 0 \\
 & 0 & 0 &  0 &}$$

Then, applying the Theorem \ref{cotadp} to the short exact sequence obtained as the middle column, we get: $$\pd(M)\leq \psi(\Omega(M'')\oplus M'\oplus P_0(M'')) + 1=\psi(\Omega(M'')\oplus M') + 1.$$

\item
Just apply 1. to the short exact sequence obtained as middle column in the last diagram.
\end{enumerate}

\end{proof}
\end{coro}

\subsection{Applications of Igusa-Todorov's Theorem to the finitistic dimension conjecture}

As a consequence of Theorem \ref{cotadp}, the main result of \cite{IT}, the finiteness of the finitistic dimension was proved for some classes of algebras.\\

We start reminding a technical Corollary of Theorem \ref{cotadp}. Consider $M \in \mod A$ with $\Llength(M) = 2$, then $\xymatrix{0 \ar[r] & T \ar[r] & P(M)/\rad^2 P(M) \ar[r] & M \ar[r] & 0}$ is exact, where $T$ is a semisimple module. As a consequence, we have the next result.
 
\begin{coro} (Corollary 1.6 of \cite{IT}) \label{ll<3} If $M$ is a f.g. $A$-module with $\Llength(M) = 2$ and finite projective dimension, then
$$\pd (M) \leq \psi(A/\rad A \oplus A/(\rad A)^2 ) + 1.$$

\end{coro}

It follows now the finiteness of the finitistic dimension for every radical cube zero algebra $A$, from Corollary \ref{ll<3} and the fact that every module in $\Omega(\mod A)$ has Loewy length less or equal to $2$.

\begin{coro}(Corollary 1.7 of \cite{IT}) Suppose that A is an Artin algebra with $(\rad A)^3 = 0$, then

$$\findim (A) \leq \psi(A/\rad A \oplus A/(\rad A)^2 ) + 2.$$

\end{coro}

But the main consequence that Igusa and Todorov obtained in \cite{IT} using their functions is the following. 

\begin{coro}(Corollary 1.9 of \cite{IT}) If $\repdim (A) \leq 3$ then $\findim (A) < \infty$.

\end{coro}

In 1994 Y. Wang used Igusa-Todorov's second function to prove in \cite{W} the following result.

\begin{teo}(Theorem 3 of \cite{W}) Suppose $A$ is a left artinian ring with $J^{2l+1} = 0$ and $\frac{A}{J^l}$ is of finite
representation type. Then  $\fin ( A ) < \infty$.

\begin{proof}
Use Corollary \ref{corocotadp}.
\end{proof}

\end{teo}

The next result, obtained by D. Xu, follows from an analogous mathematical construction.

\begin{prop}(Corollary 4.3 of \cite{Xu}) Let $A$ be a finite dimension algebra given by a quiver with relations. Suppose that $I$ is a nilpotent ideal of $A$ such that $I\rad (A) = 0$ and $A/I$
is a monomial algebra. Then $\fin (A) < \infty$.
\end{prop}

Now, let us recall the definition of Igusa-Todorov algebras introduced in \cite{We} by J. Wei.

\begin{defi}(\cite{We}) Let $A$ be an Artin algebra and $n$ be a non-negative integer. Then $A$ is said to be
$n$-Igusa-Todorov ($n$-IT for short) if there exists an $A$-module $V$ such that for any $A$-module M there is an exact
sequence
$$\delta: \xymatrix{0 \ar[r]& V_1 \ar[r]& V_2 \ar[r] & \Omega^n (M) \ar[r]& 0}$$
with $V_0, V_1 \in \add V$. When such a module exists, it is called a $n$-Igusa-Todorov module (an $n$-IT module for short).
\end{defi} 

Note that the above is equivalent to the definition given in \cite{We}.
As can be seen, this definition was given to obtain the following result, using the central theorem of \cite{IT}.

\begin{teo}\label{IT algebra}(Theorem 2.3 of \cite{We}) Let $A$ be an $n$-Igusa-Todorov algebra. Then, the finitistic dimension of $A$ is finite.

\begin{proof}

It follows from Corollary $\ref{corocotadp}$ applied to $\delta$ and the fact that $\pd (M) \leq n + \pd(\Omega^n(M))$.

\end{proof}

\end{teo}

In \cite{We} J. Wei suggests that every Artin algebra could be $n$-Igusa-Todorov, for some $n\in \N$. However, this is not true, as T. Conde observes in her doctoral dissertation \cite{T}.

\subsection{Gaps for the $\phi$ function}

Given an Artin algebra $A$, we say that a value $t \in  \mathbb{N}$, with $t \leq \fidim (A)$, is {\bf admissible} for $A$ if there exists an $A$-module M such that $\phi(M) = t$. If $t \leq \fidim (A)$ is not  admissible for $A$, we
say that there is a {\bf gap} at $t$ for $A$. Note that the notion of gap makes no sense for the finitistic dimension, i.e. if there exists a module $M$ with $0 < \pd (M) = n < \infty$, then $\pd (\Omega^i(M)) = n-i$ if $0 \leq i \leq n$.

Observe that $0$ is always an admissible value, and if $\fidim(A)$ is finite, then $\fidim(A)$, by definition, is admissible. The following theorem proves the existence of some other admissible values. The proof in \cite{BMR2} of the first statement uses the Nakayama's rule. See Theorem \ref{Nakayama}.

\begin{teo}(Theorem 4.4 of \cite{BMR2})\label{gaps}
Let $A$ be a finite dimensional algebra. If $\fidim(A) > 0$, then $1\in \mathbb{N}$ is an admissible value for $A$. If also $\fidim(A)$ is finite, then $ \fidim(A) - 1 \in \mathbb{N}$ is an admissible value for $A$.
\end{teo}

However, algebras can have gaps, as the following example shows.

\begin{ej}(Example 6.6 of \cite{BMR2})\label{ejemplo_gap} Let $A = \frac{\Bbbk Q}{J^2}$, where $Q$ is the following quiver:
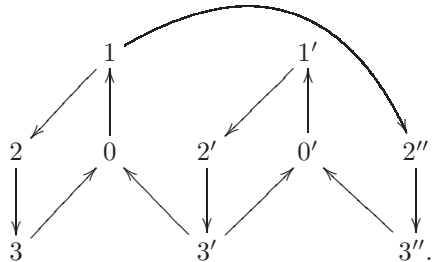
\vspace{.5cm}
$$ \xymatrix{ & 1 \ar[dl] \ar@/^3pc/[drrr] & & 1' \ar[dl] & \\ 
2 \ar[d] & 0 \ar[u] & 2'\ar[d] & 0'\ar[u] & 2''\ar[d] \\ 
3 \ar[ur]& & 3' \ar[ur] \ar[ul] & & 3''.\ar[ul]}$$
Since $S_D = S(A)$, by Corollary \ref{sinmember}, we have $\fidim(A) = \phi(\oplus_{S \in \mathcal{S}(A)} S ) + 1 = 7$. By easy computations
we have that there is no $A$-module $M$ such that $\phi (M) = 2$. On the other hand, by Theorem \ref{Rad^2izq=der}, $\fidim(A^{op}) = 7$. But $A^{op}$ has no gaps (see \cite{BMR2}). Hence, the existence of gaps is not symmetric. 

\end{ej}

A very surprising connection between these concepts and the finitistic dimension conjecture follows.

\begin{teo}(Theorem 4.1 of \cite{BMR2}) If $A$ is an Artin algebra with a gap at $t$, then the finitistic dimension conjecture holds for $A$.
\begin{proof}

Since the finitistic dimension has no gaps, using Proposition \ref{it1}, we get the result because 
$$\findim(A) < t < \fidim(A).$$
\end{proof}

\end{teo}

\section{Algebras of $\Omega^n$-finite representation type}

Let $A$ be an Artin algebra.  We say that $A$ is of {\bf $\Omega^n$-finite representation type} if $\Omega^n(\mod A)$ is of finite representation type (see Section 3.1). If $A$ is not of $\Omega^n$-finite representation type for any $n \in \mathbb{N}$ we say that $A$ is of {\bf $\Omega^{\infty}$-infinite representation type}. 

Examples of algebras of $\Omega^1$-finite representation type are:

\begin{itemize}

\item Special biserial algebras,

\item Radical square zero algebras, and

\item Truncated path algebras.

\end{itemize}

Monomial algebras are of $\Omega^2$-finite representation type (see \cite{Z}), but not of $\Omega^1$-finite representation type in general.

\begin{teo} (Theorem 3.2 of \cite{LM})
If $A$ is an Artin algebra of $\Omega^n$-finite representation type, then $\fidim(A) < \infty$ and $\psidim(A) < \infty$.

\begin{proof}

Given $M \in  \mod A$, then $\phi(M) \leq n+ \phi (\Omega^n (M))$, $\forall n \in \mathbb{N}$ by Proposition \ref{it1}. Now by Remark \ref{finitesubcategory} and the hypothesis, there exists $k \in \mathbb{N}$ such that $\phi(M) \leq n + \phi(\Omega^n(M)) \leq n+k < \infty$ for each $M \in \mod A$.

\end{proof}

\end{teo}

The following example shows that for every  $n \in \mathbb{N}$ there exists an algebra of $\Omega^n$-finite representation type but not of $\Omega^{n-1}$-finite representation type.

\begin{ej}

Let $A = \frac{\Bbbk Q}{I}$ be the finite dimensional algebra, given by $Q$ :

$$\xymatrix{ 1  \ar@/^2mm/[r]^{\alpha_1} \ar@/_2mm/[r]_{\beta_1} & 2  \ar@/^2mm/[r]^{\alpha_2} \ar@/_2mm/[r]_{\beta_2}  & 3  \ar@/^2mm/[r]^{\alpha_3} \ar@/_2mm/[r]_{\beta_3} & \cdots \ar@/^2mm/[r]^{\alpha_{n}} \ar@/_2mm/[r]_{\beta_{n}}  & n+1}$$

and $I = \langle  \alpha_i \beta_{i+1},\ \beta_i \alpha_{i+1},\ \alpha_i \alpha_{i+1} - \beta_{i} \beta_{i+1}\text{, for } i=1, \ldots, n-1  \rangle$.

Consider $M^k_{\mu,\lambda}$ the following $A$-modules for $k = 1, \ldots,n$, $\lambda, \mu \in \Bbbk $ and $\topp(M^k_{\mu,\lambda}) = S_k$.

$$M^k_{\mu, \lambda} = \xymatrix{ 0  \ar@/^2mm/[r]^{0} \ar@/_2mm/[r]_{0} & 0  \ar@/^2mm/[r]^{0} \ar@/_2mm/[r]_{0}  & \cdots  \ar@/^2mm/[r]^{0} \ar@/_2mm/[r]_{0} & \Bbbk  \ar@/^2mm/[r]^{ \mu 1_{\Bbbk}} \ar@/_2mm/[r]_{\lambda 1_{\Bbbk}} & \Bbbk \ar@/^2mm/[r]^{0} \ar@/_2mm/[r]_{0} & \cdots \ar@/^2mm/[r]^{0} \ar@/_2mm/[r]_{0}  & 0 \ar@/^2mm/[r]^{0} \ar@/_2mm/[r]_{0} & 0}.$$

Then $\Omega(M^{k}_{1,\lambda}) \cong M^{k+1}_{-\lambda,1}$ and $\Omega(M^{k}_{\lambda, 1}) \cong M^{k+1}_{1,-\lambda}$ for $k = 1, \ldots, n-1$, $\lambda \in \Bbbk$ and $\Omega(M^n_{\mu, \lambda}) = P_{n+1}$ for $\lambda, \mu \in \mathbb{\Bbbk}$. Hence $A$ is of $\Omega^{n+1}$-finite representation type but not of $\Omega^n$-finite representation type. 
\end{ej}

\subsection{Radical square zero algebras}

Given a finite quiver $Q$, with $\vert Q_0 \vert = n$, we call the {\bf heart} of $Q$ the full
subquiver of $Q$ determined by the vertices in the intersection of the support of $\Omega^n (A_0)$
with the support of $\Omega^{-n} (A_0)$, we denote it by $C(Q)$. We call {\bf member} of $Q$ the full subquiver of $Q$ determined by the vertices that are not in the heart. We denote it by $M(Q)$. Throughout this section, we work with radical square zero algebras $A = \frac{\Bbbk Q}{J^2}$ with $\vert Q_0 \vert= n$. Recall that $\mathcal{S}_D = \mathcal{S}\setminus (\mathcal{S}_P \cup \mathcal{S}_I)$, that is the simple nonprojective and noninjective modules.\\

It is known that the finitistic dimension of radical square zero algebras is finite. The $\phi$-dimension and $\psi$-dimension are also finite for algebras in this class. We recall some results obtained in \cite{LMM}.

\begin{prop}(\cite{LMM})\label{toto}\label{explicito}\label{sinmember}
If $A = \frac{\Bbbk Q}{J^2}$, then

\begin{enumerate}

\item (Propositions 4.12 and 4.16) $\fidim(A) \leq \phi(A_0) + 1 \leq \vert Q_0 \vert = n$.

\item (Proposition 4.14) $\fidim(A) = \phi(\oplus_{S\in \mathcal{S}_D}S)+1$, in case $A$ is not selfinjective. 

\item If $n=1$, then $\psi \dim (A) \leq \vert Q_0 \vert = 1$, otherwise $\psi \dim (A) \leq 2\vert Q_0 \vert-2$.

\item (Proposition 4.63) If $\phi \dim (A) = \vert Q_0 \vert$, then $\psi \dim (A) = \vert Q_0\vert$.

\end{enumerate}

\end{prop}

%
%
%

The example below shows a finite dimensional algebra $A$ where the previous bound is reached. The computation of its $\phi$-dimension is left to the reader.

\begin{ej}\label{ejemplo_maximal}
Let $A = \frac{\Bbbk Q}{J^2}$ be the finite dimensional algebra, given by Q:

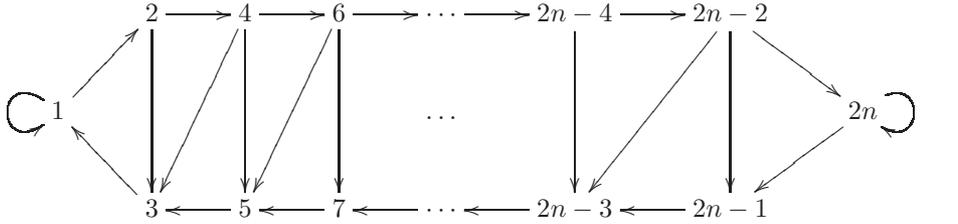
$$\xymatrix{ & 2 \ar[r] \ar[dd] & 4 \ar[r] \ar[dd] \ar[ldd]& 6 \ar[r] \ar[dd] \ar[ldd]  & \cdots \ar[r]  & 2n-4 \ar[dd]\ar[r] & 2n-2 \ar[dd] \ar[dr] \ar[ldd]  & &  \\ 1  \ar@(ul, dl)   \ar[ur] & & &  & \ldots & & & 2n \ar[dl]  \ar@(ur, dr) & .   \\ & 3 \ar[lu] & 5 \ar[l]  & 7 \ar[l] & \cdots \ar[l] & 2n-3 \ar[l] & 2n-1\ar[l]  &  & }$$ 
 In this case $\fidim(A) = 2n = \vert Q_0\vert$.
\end{ej}
 
In case not all the vertices were at the heart of $Q$, the bound is

\begin{prop} \label{SD}(Proposition 4.19 of \cite{LMM})
If $A = \frac{\Bbbk Q}{J^2}$ and $M(Q)$ is not empty, then
$$\fidim(A) \leq \vert \mathcal{S}_D \vert < \vert Q_0 \vert = n.$$
\end{prop}

\begin{coro}\label{otra_cota}
If $A = \frac{\Bbbk Q}{J^2}$, then 
$$\fidim (A) \leq \frac{\dim_{\Bbbk}(A) -1}{2}.$$

\begin{proof}
We divide the proof in two cases, if $\dim_{\Bbbk} A$ is even or odd.

If $\dim_{\Bbbk} A = 2m+1$, then $|Q_0| \leq m+1$, and by Proposition \ref{toto} follows that $\fidim (A) \leq |Q_0| \leq m+1$. In case $ |Q_0| = m+1 $ the quiver $Q$ is a tree, then $\fidim (A) = \gl(A) \leq m$. 

If $\dim_{\Bbbk} A = 2m$, then $\fidim(A) \leq |Q_0| \leq m$. If $|Q_0| = m$, then $Q = C^m$ or $M(Q)$ is not empty.
We know that $\fidim(\frac{\Bbbk C^m}{J^2}) = 0$ (see Proposition \ref{autoinyectiva}). In the other case, by Proposition \ref{SD}, $\fidim (A) < m$, and it follows the thesis.
\end{proof}

\end{coro}

Note that if $Q$ is an oriented $A_n$, then the bound in Corollary \ref{otra_cota} is reached.

\subsection{Truncated path algebras}

Throughout this section, we work with a truncated path algebras $A = \frac{\Bbbk Q}{J^k}$ for some $k \geq 2$, with $\vert Q_0 \vert =n$. This class of algebras is a natural extension of the class of radical square zero algebras. The following result is a generalization of the second item of Proposition \ref{explicito}.

\begin{prop}(Corollary 1 of \cite{BMR})\label{phitruncadas}
If $A$ is not a selfinjective algebra, then 
$$\fidim(A) = 1+\phi(\bigoplus_{ l(\rho) \geq 1} \rho A).$$
\end{prop}

For a truncated path algebra $A$, we denote by $M_v^l (A)$ the ideal $\rho A$, where $\length(\rho) = l$, $\target (\rho) = v$ and $ M_l (A) =
\displaystyle\bigoplus_{v \in Q_0} M_v^l (A)$.

\begin{prop}\label{autoinyectiva}(Proposition 4 of \cite{BMR}) For a non-simple algebra $A = \frac{\Bbbk Q}{J^k}$ the following statements are equivalent:
\begin{enumerate}
\item $M^l_v$ exists, is not a projective module and $\Omega(M^l_v)$ is indecomposable for all $v \in Q_0$ and $1 \leq l \leq k-1$.
\item $Q$ is a cycle, i.e., $Q = C^n$.
\item $A$ is a selfinjective algebra.
\item The $\fidim (A) = 0$.
\item The $\psidim (A) = 0$.
\end{enumerate}
\end{prop}

Given a truncated path algebra $A = \frac{\Bbbk Q}{J^k}$, adding a loop at each sink and at each source of $Q$ we obtain a new quiver $\bar{Q}$ with every vertex in its heart, such that the truncated path algebra $B = \frac{\Bbbk \bar{Q}}{J^k}$ verifies $\fidim(A) \leq \fidim(B)$.

\begin{teo}\label{phi maximas}(Theorem 4 of \cite{BMR})
If $A = \frac{\Bbbk Q}{J^k}$, then there exist a truncated path algebra $B = \frac{\Bbbk \bar{Q}}{J^k}$ such that $\fidim(A) \leq \fidim(B)$, where $\bar{Q} = (\bar{Q}_0, \bar{Q}_1)$ is a quiver, such that $Q_0 = \bar{Q}_0$ and $Q_1 \subseteq \bar{Q}_1$ with $M(\bar{Q})$ empty. 
\end{teo}

\begin{defi}
We define $f_k:\N\rightarrow\N$, with $k \in \mathbb{N}, k\geq 2,$ as
\[
f_k(m)=
\left\{
\begin{array}{ll}
0 & \text{if}\ m=0,\\
2\left(\frac{m-1}{k}\right)+1 & \text{if}\ m\equiv1\pmod{k},\\
2\left(\frac{m-2}{k}\right)+2 & \text{if}\ m\equiv2\pmod{k},\\
2\left\lceil\frac{m-2}{k}\right\rceil+1 & \text{otherwise}.
\end{array}
\right.
\]
\end{defi}

Note that $f_k$ is an increasing function.
The following theorem shows that the $\phi$-dimension of a truncated path algebra can be calculated from the $\phi$-dimension of the associated radical square zero algebra using the function defined above.

\begin{teo}\label{phi sin pozos ni fuentes}(Theorem 5 of \cite{BMR})
Consider $A = \frac{\Bbbk Q}{J^k}$ a non selfinjective algebra with $M(Q)$ empty, then
$$ \fidim (A) = f_k(\fidim(\frac{\Bbbk Q}{J^2}))$$
\end{teo}

The next result follows from Theorem \ref{phi maximas}, Theorem \ref{phi sin pozos ni fuentes}, and Proposition \ref{toto}.

\begin{coro}\label{boundcoro}(Corollary 2 of \cite{BMR}) If $A = \frac{\Bbbk Q}{J^k}$, then $\fidim (A) \leq f_k(n)$
and the bound is reached.
\begin{proof}

By Theorem \ref{phi maximas}, there is a truncated path algebra $B = \frac{\Bbbk \bar{Q}}{j^k}$ with $\bar{Q}_0 = Q_0$, $Q_1 \subset \bar{Q}_1$ and $M(\bar{Q})$ empty such that $\fidim (A) \leq \fidim (B)$. By Theorem \ref{phi sin pozos ni fuentes} $\fidim (B) \leq f_k (\fidim(\frac{\Bbbk \bar{Q}}{J^2}))$ and since $f_k$ is an increasing function, then $f_k(\fidim(\frac{\Bbbk \bar{Q}}{J^2})) \leq f_k(n)$ and the thesis follows. Note that $\fidim (\frac{\Bbbk Q}{j^k}) = f_k(n)$ if $Q$ is the quiver of Example \ref{ejemplo_maximal} 
\end{proof}
\end{coro}

\subsection{Monomial algebras}

Monomial algebras give shape to a family of extensively studied Artin algebras. They play a central role in the evolution of the theory. For a monomial algebra $A$ given by a quiver Q, we recall that $\mathbb{P}^{\geq 1}$ is the set of paths of length greater or equal to one. It is clear that the subset of nonzero paths of $A$ is a finite set. 

The following proposition shows that the value of $\phi$ on the right ideals bounds the $\phi$-dimension on monomial algebras. Later remarks and example shows that the three values can be reached.  

\begin{prop} \label{phi monomial}

Let $A = \frac{\Bbbk Q}{I}$ be a monomial algebra, then
$$\phi( \bigoplus_{\rho \in \mathbb{P}^{\geq 1}} \rho A) \leq \fidim (A) \leq \phi( \bigoplus_{\rho \in \mathbb{P}^{\geq 1}} {\rho} A) + 2$$

\begin{proof}

We know that $\fidim(A)\leq \fidim (\Omega^2(\mod A)) +2$. By Theorem I of \cite{Z}, $\Omega^2(\mod A) \subset \add{(\oplus_{\rho \in \mathbb{P}^{\geq 1} } \rho A)}$, then $\fidim (A) \leq \phi( \oplus_{\rho \in \mathbb{P}^{\geq 1}} {\rho} A) + 2$.

\end{proof}

\end{prop}

The next remark and the subsequent example show that the three different values of the proposition above can be reached.

\begin{obs} Let $A$ be an monomial algebra, then

\begin{itemize} 

\item $A$ is selfinjective then $\phi( \oplus_{\rho \in \mathbb{P}^{\geq 1}} \rho A) = \fidim (A) = 0$.

\item If $A$ is a non selfinjective truncated path algebra, then $\fidim(A) = \phi( \oplus_{\rho \in \mathbb{P}^{\geq 1}} \rho A) +1$ (see Proposition \ref{phitruncadas})
\end{itemize} 

\end{obs}

The next example shows that there are monomial algebras $A$ such that $ \fidim (A) = \phi( \oplus_{\rho \in \mathbb{P}^{\geq 1}} \rho A) +2$.

\begin{ej}

Consider $A = \frac{\Bbbk Q}{I}$, where $Q$ is given by

$$
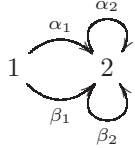
\xymatrix{ 1  \ar@/^4mm/[r]^{{\alpha}_1}   \ar@/_4mm/[r]_{{\beta}_1}  & 2 \ar@(lu,ru)^{\alpha_2} \ar@(ld,rd)_{\beta_2} }$$ 
and $ I  = \langle J^3,  \alpha_i \beta_2,  \beta_i \alpha_2 \text{ for } i =1,2 \rangle $. 
Since $\alpha_2^2 A = \beta_2^2 A = \alpha_1  \alpha_2 A =  \beta_1 \beta_2 A = S_2 $, $\alpha_1 A = \alpha_2 A$, $\beta_1 A = \beta_2 A$, and 

\begin{itemize}

\item $\Omega (S_2) = \alpha_2 A \oplus \beta_2 A$,

\item $\Omega (\alpha_2 A ) = \beta_2 A \oplus S_2$,

\item $\Omega (\beta_2 A )= \alpha_2 A\oplus S_2$,

\end{itemize}
we have that $\phi(\oplus_{\rho \in  \mathbb{P}^{\geq 1}}\rho A) = 0$ and $\fidim (A) \leq 2$. On the other hand, if we consider the A-module $M$ as follows

$$
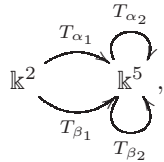
\xymatrix{ \Bbbk^2  \ar@/^4mm/[r]^{ T_{\alpha_1}}   \ar@/_4mm/[r]_{T_{\beta_1}}  & \Bbbk^5 \ar@(lu,ru)^{T_{\alpha_2}} \ar@(ld,rd)_{T_{\beta_2}} },$$  
where $\langle e_1, e_2\rangle $  and $\langle f_0, f_1, f_2, f_3, f_4 \rangle$ are basis of $\Bbbk^2$ and $\Bbbk^5$ respectively and

\begin{itemize}

\item $T_{\alpha_1}(e_1) = f_1$ and $T_{\alpha_1}(e_2) = f_2$,

\item $T_{\beta_1}(e_1) = f_2 $ and $T_{\beta_1}(e_2) = f_3$,

\item $T_{\alpha_2}(f_0) = 0, $ $T_{\alpha_2}(f_1) = f_0 $, $T_{\alpha_2}(f_2) = 0 $, $T_{\alpha_2}(f_3)= 0$ and $T_{\alpha_2}(f_4)= 0$,

\item $T_{\beta_2}(f_0) = 0$, $T_{\beta_2}(f_1) = 0$, $T_{\beta_2}(f_2) = 0$, $T_{\beta_2}(f_3) = f_4$ and $T_{\beta_2}(f_4) = 0$.

\end{itemize}
Hence $\Omega (M) = N$, where $N$ is represented by:

$$
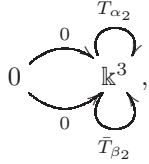
\xymatrix{ 0 \ar@/^4mm/[r]^{ 0 }   \ar@/_4mm/[r]_{0}  & \Bbbk^3 \ar@(lu,ru)^{\bar{T}_{\alpha_2}} \ar@(ld,rd)_{\bar{T}_{\beta_2}} },$$  
where $ \langle g_1, g_2, g_3 \rangle $ is a basis of $\Bbbk^3$

\begin{itemize}

\item $\bar{T}_{\alpha_2} (g_1) = g_2$, $\bar{T}_{\alpha_2}(g_2) = 0 $ and $\bar{T}_{\alpha_2}(g_3) = 0$,

\item $\bar{T}_{\beta_2} (g_1) = g_3$, $\bar{T}_{\beta_2}(g_2) = 0 $ and $\bar{T}_{\beta_2}(g_3) = 0$.

\end{itemize}
Therefore $N$ is an $A$-module that does not belong to $\Omega^2(\mod A) \subset \add{(\oplus_{\rho \in \mathbb{P}^{\geq 1} } \rho A)}$, thus $$\phi(M \oplus \alpha_2 A \oplus \beta_2 A \oplus S_2) = 2 = \fidim (A).$$
\end{ej}

For monomial algebras $A=\Bbbk Q/I$ in \cite{HZ} and \cite{IZ}, respectively,  the authors present bounds for the global dimension of $A$: 
$$
\gl (\Lambda) \leq  \dim_{\Bbbk} \Lambda,$$
\noindent or the stronger statement:
 $$\gl (\Lambda) \leq  \dim_{\Bbbk} \rad \Lambda.$$ 

We already observed that, if the algebra has finite global dimension -as monomial algebras- then $\gl$ and $\fidim$ coincide. In this sense, the next is another bound for monomial algebras, obtained from Proposition \ref{phi
 monomial}.

\begin{teo} (Corollary 3.9 of \cite{LM}) If $A = \Bbbk Q/I$ is a monomial algebra, then $\fidim (A) \leq {\dim }_{\Bbbk}A-\vert Q_0 \vert+2$.

\begin{proof}

First note that $\dim_{\Bbbk} A = \vert \mathbb{P}^{\geq 1} \vert + \vert Q_0 \vert$ and apply Proposition \ref{phi monomial}, finally apply Proposition \ref{invariante} to $\oplus_{\rho \in \mathbb{P}^{\geq 1}} \rho A$. 
\end{proof}

\end{teo}

\subsection{Geometric aspects on monomial algebras}

Now we present some properties of the Igusa-Todorov $\phi$ function that can be deduced from the quiver in specific subclasses of monomial algebras.

\begin{prop}
Let $A = \frac{\Bbbk Q}{I}$ be a monomial algebra. If $\fidim (A) = 0$, then $Q = C^m$ and $I = J^k$.

\begin{proof}

Note, by Theorem \ref{selfinjective}, that $A$ is a selfinjective algebra. If there is a vertex $v \in Q_0$ with starting (arriving) degree different to $1$, by Theorem \ref{Nakayama} the rule $\nu$ is not a permutation and we deduce that $Q = C^m$.\\

If $I \not= J^k$ for any $k$, there are two projective modules $P_i$ and $P_j$ with $\dim_{\Bbbk} P_i \not =  \dim_{\Bbbk} P_j$. We can suppose that $j = i+1$.
\begin{enumerate}
\item Suppose $\dim_{\Bbbk} P_i > \dim_{\Bbbk} P_{i+1}$. There are two possible cases

\begin{itemize}

\item If $\soc (P_{i}) = \soc (P_{i+1})$, there are a short exact sequence 
$$ \xymatrix{0 \ar[r] & P_{i+1} \ar[r] & P_{i} \ar[r] & M \ar[r] & 0.} $$
Since $P_{i+1}$ is an injective module. The short exact sequence splits, hence $P_{i+1}\oplus M \cong P_{i}$, which is absurd. 

\item If $ \soc (P_{i}) = S_j \not = \soc (P_{i+1}) = S_{l}$ with $l<j$, then there is a nonzero path $\rho : i \rightarrow j$ and zero path $\lambda : i+1 \rightarrow l+1$ which is absurd.

\end{itemize}

\item Suppose now $\dim_{\Bbbk} P_i < \dim_{\Bbbk} P_{i+1}$. Then $P_i = I_l$ and $P_{i+1} = I_{l+s}$ with $s\geq 2$. Consider the injective module $I_{l+1}$. Since $A$ is selfinjective $I_{l+1} = P$, hence $P$ must be $P_{i}$ or $P_{i+1}$. Both cases are not possible by the dual argument of part $1.$ with $\soc (P_{i}) = \soc (P_{i+1})$.  

\end{enumerate}

\end{proof}
 
\end{prop}

The following example shows a monomial algebra whose quiver is the oriented $C_n$ with positive global dimension. 

\begin{ej}(\cite{HZ})
Let $Q$ be an oriented cycle with $n$ vertices and $n$ arrows $\alpha_i$ for $1 \leq i \leq n$ such that $\start(\alpha_i ) = i = \target(\alpha_{i-1})$ for $1 < i < n$ and $\target(\alpha_n ) = 1 = \start(\alpha_1)$.
$$\xymatrix{ & 1 \ar[r]^{\alpha_1} & 2 \ar[rd]^{\alpha_2} & \\ n \ar[ur]^{\alpha_n} & & & 3 \ar[d]  \\ \vdots \ar[u] &&& \vdots \ar[d] \\ m+3 \ar[u] &&& m \ar[dl]^{\alpha_{m}} \\ & m+2 \ar[lu]^{\alpha_{m+2}} & m+1 \ar[l]^{\alpha_{m+1}}& }$$
Let $w_1 = \alpha_1 \alpha_2 \ldots \alpha_{n-1} \alpha_n$ and for $2 \leq i \leq n-1$ let $w_i = \alpha_i \alpha_{i+1} \ldots \alpha_{i-1} \alpha_i$.
Consider the monomial algebra $\frac{\Bbbk Q}{I}$, where $I = \langle w_1, \ldots, w_{n-1} \rangle$. Then $\pd(S_{n-i}) = 2i+1$  for $0 \leq i \leq n-1$ and $\pd (S_1) = 2n-2$.
\end{ej}   

The following result follows directly from Theorem \ref{phi sin pozos ni fuentes} and Corollary 4.62 of \cite{LMM}.

\begin{coro} 
Let $A = \frac{\Bbbk Q}{J^m}$ be a truncated path algebra. If $Q$ is symmetric, then $\fidim (A) \leq 2$.

\end{coro}

\begin{coro}(Corollary 4.40 of \cite{LMM}) \label{2loops}
Let $A = \frac{\Bbbk Q}{J^2}$ with $\vert Q_0 \vert = n$. If $\fidim (A) = n$, then $A$ is strongly connected, and $Q$ has at least two loops. 
\end{coro}

\begin{prop}
Let $A = \frac{\Bbbk Q}{J^2}$. If $\fidim (A) = \frac{\dim_{\Bbbk}(A) -1}{2}$, then $Q$ is an oriented $A_n$.

\begin{proof}
Since $Q$ is connected, $ \vert Q_1 \vert \geq \vert Q_0\vert -1 = n-1$. Then $\frac{\dim_{\Bbbk}(A) -1}{2} \geq n-1$ and we have two cases: $\fidim (A) = n \text{ or } n-1$.

\begin{itemize}

\item If $\fidim (A) = n$, then by Proposition \ref{2loops} $Q$ is strongly connected and has two loops, so we have $\frac{\dim_{\Bbbk}(A) -1}{2} > n$. 

\item Consider now $\fidim (A) = n-1$. If  $\fidim (A) = \frac{\dim_{\Bbbk}(A) -1}{2}$, then $\vert Q_1 \vert = n-1$ and we deduce that $Q$ is an oriented $A_n$.
\end{itemize}
\end{proof}
\end{prop}

For a proof of the following result, see Theorem 6.10 of \cite{BMR2}

\begin{prop}(Theorem 6.10 of \cite{BMR2})
Let $Q$ be a quiver such that one of the following conditions holds.

\begin{enumerate}
\item Every vertex of $Q_0$ has outdegree bigger or equal to $6$ and for two different vertices of $Q_0$, $u$ and $u'$, there exist $4$ different arrows $\alpha, \alpha', \beta, \beta'$ such that
$$\start(\alpha) = \start(\alpha') = u,\  \start(\beta) = \start(\beta') = u',\ \target(\alpha) = \target(\beta)\ \text{and }\target(\alpha') = \target(\beta').$$

\item every vertex $v \in Q_0$ has a double arrow $\alpha_v$, $\beta_v$. Aditionally, for two different vertices of $Q_0$, $u$ and $u'$, there exist $4$ different arrows $\alpha, \alpha', \beta, \beta'$ such that
$$\start(\alpha) = \start(\alpha') = u,\ \start(\beta) = \start(\beta') = u',\ \target(\alpha) = \target(\beta)\text{ and }\target(\alpha') = \target(\beta'),$$
$$ \{\alpha, \alpha'\}\cap\{\alpha_u, \beta_u\}=\emptyset,\ \{\beta, \beta'\}\cap\{\alpha_{u'}, \beta_{u'}\} = \emptyset.$$ 

\end{enumerate}

Then $A$ has no gaps.

\end{prop}

\section{Selfinjective algebras, Frobenius extensions, and Gorenstein algebras}

Both functions, $\phi$ and $\psi$, were very useful for classifying selfinjective algebras. In \cite{HL}, was proved that selfinjective algebras are precisely those for which $\fidim$ and $\psidim$ vanish. On the other hand, selfinjective algebras are just $0$-Gorenstein algebras (see definition below). In \cite{LM} and \cite{GS} the authors showed independently that both $\fidim$ and $\psidim$ of $m$-Gorenstein algebras are equal to $m$.

\begin{teo} (Theorem 6 of \cite{HL})\label{selfinjective} For a right artinian ring $R$, the following are equivalent.

\begin{itemize}
\item $\fidim (R) = 0$.
\item $\psidim (R) = 0$.
\item $R$ is right selfinjective.
\end{itemize}
\end{teo}

\begin{coro} (Corollary 7 of \cite{HL})
Let $A$ be an Artin algebra. Then $A$ is selfinjective if and only if
$\fidim(A) = 0$.
\end{coro}

Auslander and Bridger defined stable modules in \cite{AB}:
$$^\bot A =\{M\in \mod A:\ \ext^i(M,A)=0, \ \forall i\geq 1\}.$$ 

For a selfinjective algebra, since $A$ is injective, all the modules are stable, that is: $^\bot A =\mod A$. For any Artin algebra, the next theorem shows that both IT-functions vanish on stable modules. The main argument follows by the fact that $^\bot A$ is invariant by $\Omega$, and the functor $\Omega: \underline{\mod}A\rightarrow \underline{\mod}A$ is bijective when restricted to $^\bot A$, see \cite{AB}.

\begin{teo}\label{phi en perpA} (Theorem 3.13 of \cite{LM})
For any Artin algebra $A$, the Igusa-Todorov functions vanish over $\ ^\bot A$, that is,  
$$\fidim( ^\bot A) = \psidim ( ^\bot A) = 0.$$
\end{teo}

It is known that Gorenstein projective modules ($\mathcal{G} P_A$), are stable modules. 
This implies the next result.

\begin{coro}(Corollary 4.1 of \cite{LM})
The Igusa-Todorov functions vanish over $\mathcal{G} P_A$, that is, 
$$\fidim( \mathcal{G} P_A) = \psidim (\mathcal{G} P_A) = 0.$$
\end{coro}

By Equation \ref{desigualdades} and Theorem \ref{phi en perpA}, we have the following.
\begin{coro}(Corollary 3.17 of \cite{LM})
For any Artin algebra $A$ such that $\id (A_A) = n < \infty$, it follows that 
$$\fin (A) \leq \fidim (A) \leq \psidim (A) \leq n.$$
\begin{proof}
Given $M \in \mod A$, since $\ext^k_A(\Omega^n(M),A) = \ext^{n+k}_A(M,A) = \ext^k_A(M,\Omega^{-n}(A)) = 0$ for $k\geq 1$, we have that $\Omega^n(M) \in ^{\bot}{\hspace{-5pt}}A$. By Theorem \ref{phi en perpA},
$\phi(\Omega^n (M)) = \psi(\Omega^n (M)) = 0$. Therefore by Propositions \ref{it1}, we obtain that $\phi(M) \leq
\psi(M) \leq n$.
\end{proof}

\end{coro}

An Artin algebra $A$ is called a {\bf m-Gorenstein algebra} if $\id (A_A) = m < \infty$ and $\id ( _AA) = m < \infty$. A selfinjective algebra $A$ is a particular case of Gorenstein algebra because $\id (A_A) = \id ( _AA) = 0$, that is selfinjective algebras are the 0-Gorenstein algebras. When for self-injecive algebras the $\phi$-dimension vanish, for a Gorenstein $A$ it can be proved that $\fin (A) = \fidim (A) = \psidim (A) < \infty$. More precisely

\begin{teo}\label{phi algGorenstein}\label{psi algGorenstein} (Theorem 2.7 of \cite{GS} and Theorem 4.7 of \cite{LM})
If $A$ is a $m$-Gorenstein algebra, then 
$$\fin (A) = \fidim (A) = \psidim (A)  = m.$$
\end{teo}

\section{Characterization of the IT-function $\phi$ by the bifunctors Ext and Tor}

We claim that the functions $\phi$ and $\psi$ are accurate homological measures of the Representation theory of Artin algebras.
For this claim, it is necessary to exhibit a close connection between IT-functions and the classical homological bifunctors Ext$_A^i(\ - \ ,\ - \ )$ and Tor$_i^A(\ - \ , \ - \ )$.\\

Furthermore, by using the characterization
of the $\phi$-function by the bifunctors, in \cite{FLM} the authors showed that the finiteness of the $\phi$-dimension of an Artin algebra is invariant under derived equivalences. As an application of the results obtained in  \cite{FLM}, the authors generalized the classical Bongartz's theorem about the tilting process (\cite{B}), but also the result of Keller (\cite{GR}) and Kato (\cite{Ka}) about global dimensions of derived equivalent algebras.

\subsection{d-divisions}
Along this subsection, we recall the concept of d-division, a basic notion to exhibit the connection between the IT-function $\phi$ and the bifunctors Ext$_A^i(\ - \ ,\ - \ )$ and Tor$_i^A(\ - \ ,\ -\ )$.\\ 

After that, we show some equivalent ways to define the Igusa-Todorov $\phi$-function. We now introduce the notion of $d$-division. This definition is slightly different but equivalent to the original one (\cite{FLM}).

\begin{defi}(\cite{FLM})\label{E division}
Let $A$ be an Artin algebra, $d$ be a positive integer and $M$ in
$\mod A $. A pair $(X, Y )$ of objects in $\add(M)$ is called a $(d, E)-division$ of $M$ if the
following two conditions hold:

\begin{itemize}


\item $\ext^d_A (X,\ - \ ) \not\cong \ext^d_A (Y,\ - \ )$ in $\mathcal{C}_A$,

\item $\ext^{d+1}_A (X,\ - \ ) \cong \ext^{d+1}_A (Y,\ - \ )$ in $\mathcal{C}_A$.

\end{itemize}








\end{defi}
















The following theorem gives a characterization of the  function $\phi$ in $\mod A$ in terms of the bifunctor Ext$_A^i(\ - \ ,\ - \ )$.  

\begin{teo}(Theorem 3.6 of \cite{FLM})\label{teorema D-division} Let $A$ be an Artin algebra and $M$ in $\mod A$. Then

$$ \begin{array}{ccc}
\phi(M) & = & \max(\{d \in \N : \text{ there is a }(d, E)-division \text{ of } M \} \cup \{0\}) 
\end{array} $$

\end{teo}

The previous notion of D-division can also be defined using the functor Ext on its first variable. In this case it gives the $\phi$-dimension of the opposite algebra. These dimensions can also be defined using the Tor functor in both variables.

%
%
%
%
%
%
%
%
%
%
%

\subsection{$\phi$-dimension and derived equivalent algebras}

Using the characterization of the $\phi$ function by the bifunctor, Theorem \ref{teorema D-division}, in \cite{FLM} the authors showed the the finiteness of the $\phi$-dimension of an Artin algebra is invariant under derived equivalences.

\begin{teo}(Theorem 4.10 of \cite{FLM}) \label{derivadamente eq} Let $A$ and $B$ be Artin algebras, which are derived equivalent. Then,
$\fidim (A) < \infty$ if and only if $\fidim (B) < \infty$. More precisely, if $T^{\bullet}$ is a tilting
complex over A with $n$ non-zero terms and such that $B \simeq \enn_{D(A)} (T^{\bullet})$, then
$$\fidim (A)-n \leq \fidim (B) \leq \fidim (A) + n.$$

\end{teo}

\begin{coro}(Corollary 4.11 of \cite{FLM})\label{derived eq FLM} Let $A$ be an Artin algebra, and let $T \in \mod A$ be a tilting $A$-module. Then, for the Artin algebra $B \simeq \enn_A (T)$, we have that
$$\fidim (A)-\pd (T) \leq \fidim (B) \leq \fidim (A) + \pd (T).$$

\end{coro}

\begin{obs} 
\begin{enumerate}
\item In \cite{GR} and \cite{Ka}, Keller and Kato, respectively, showed that the difference of the global dimensions of two derived equivalent algebras $A$ and $B$ is at most $n$ (the length of the tilting complex). This result can be seen as a consequence of Corollary \ref{derived eq FLM}.  
\item In Corollary 1 of \cite{B}, Bongartz considers a classical tilting module, that is, a tilting module $T \in \mod A$ such that $\pd (T) \leq 1$.
The original Bongartz's result says that $\gl(B) \leq \gl(A) + 1$, where $B =
\enn_A (T)^{op}$. It can be seen as a particular case of Theorem \ref{derivadamente eq}.
\end{enumerate}

\end{obs}

\section{Idempotent ideals, recollements, and Igusa-Todorov functions}

\subsection{Idempotent ideals}

We start by recalling some definitions which will be used throughout this section. We recall that an ideal $\textfrak{I}$ is a strong idempotent ideal if the morphism $\ext_{A/\textfrak{I}}^i(X,Y) \rightarrow \ext_{A}^i(X, Y) $ induced by the canonical isomorphism $\hom_{A/\textfrak{I}} (X, Y ) \rightarrow \hom_A (X, Y )$ is an isomorphism for all $i \geq 0$ and all $X, Y$ in $\mod A/\textfrak{I}$. Let $\textfrak{I}$ be a strong idempotent ideal of $A$, $P_0$ the projective cover of $\textfrak{I}$ and $P = eA$ where $e$ is an idempotent element of $A$ such that $\add P = \add P_0$. It is known that  $\mod A/\mathfrak{I}$ is a Serre subcategory of $\mod A$, and this inclusion induces an exact sequence of categories 
$$\xymatrix{ \mod A/\mathfrak{I} \ar[r]& \mod A \ar[r]^{e_P} & \mod \Gamma,}$$
where $\Gamma = End_A (P )^{op}$ and $e_P = \hom_A(P,\ - \ )$.

\begin{teo}(Theorem 4.3 of \cite{GLP}) Let $\textfrak{I}$ be a strong idempotent ideal of $A$ such that $\pd_A(A/\textfrak{I}) < \infty$. Then

\begin{enumerate}

\item $\phi_{A/\textfrak{I}}(X) \leq \phi_{A}(X)$ for all $X \in \mod A/\textfrak{I}$.

\item $\fidim (A/\textfrak{I}) \leq \fidim (A)$.

\end{enumerate}

\end{teo}

Let $P/\rad P \simeq S_1 \oplus \cdots \oplus S_r $, with $S_i$ simple for all $i$, so that $P = P_0 (S_1 \oplus \cdots \oplus S_r)$ and $I = I_0 (S_1 \oplus \cdots \oplus S_r)$. Consider $\mathbb{P}_k$ the full subcategory of $\mod A$ consisting of the $A$-modules $X$ having a projective resolution $\xymatrix{ \cdots \ar[r] & P_1 \ar[r] & P_0 \ar[r] & X \ar[r] & 0}$ with $P_i$ in $\add P$ for $0 \leq i \leq k$, and $\mathbb{P}_{\infty} =\displaystyle\bigcap_{k \geq 0} \mathbb{P}_k$.  The full subcategories $\mathbb{I}_k$ and $\mathbb{I}_{\infty}$ are defined dually. 

To compare the behavior of the Igusa-Todorov functions for $\mod A$ and $\mod \Gamma$, the authors proved in \cite{GLP} that both functions are preserved under $e_P$ for $A$-modules having a resolution in $\add P$. As a consequence, they proved:

\begin{teo}(Theorem 4.7 of \cite{GLP})
\begin{itemize}

\item If $\mathbb{P}_1 = \mathbb{P}_{\infty}$ then $\fidim (\Gamma) \leq \fidim (A)$ and $ \psidim (\Gamma) \leq
 \psidim (A)$.

\item If $\mathbb{I}_1 = \mathbb{I}_{\infty}$ then $\fidim (\Gamma^{op}) \leq \fidim (A^{op})$ and $ \psidim (\Gamma^{op}) \leq
 \psidim (A^{op})$.

\end{itemize}

\end{teo}

Finally, the authors related the homological dimensions (including $\phi$-dimension) of the algebras involved.

\begin{prop}(Proposition 4.10 of \cite{GLP}) Let $A$ be an Artin algebra and $\mathfrak{I}$ be a strong idempotent ideal. Then
$$\fidim (A^{op}) \leq \pd_{A} (A/\mathfrak{I}) + \max\{ \gl( \Gamma) +1, \pd_{A^{op}}(A/\mathfrak{I}) + \fidim ((A/\mathfrak{I})^{op}) \}.$$

\end{prop}

Observe now that when $\mathfrak{I}$ is a strong idempotent ideal then $\mathfrak{I}$ is in $\mathbb{P}_{\infty}$ (\cite{APT},
Theorem 2.1), so $\pd_A (M) \leq \pd_{\Gamma} (Hom_A (P, M)) \leq \gl \Gamma$. Thus $\pd_A (A/\mathfrak{I}) \leq \gl( \Gamma) +1$. Since being a strong idempotent ideal is a symmetric condition, we obtain that $\pd((A/\mathfrak{I})^{op}) \leq \gl (\Gamma ) + 1$, as observed in \cite{APT} at the end of Section 5.

\begin{coro}(Corollary 4.11 of \cite{GLP}) Let $A$ be an Artin algebra and $\mathfrak{I}$ be a strong idempotent ideal. Assume $\gl(\Gamma) < \infty$, then $\fidim((A/\mathfrak{I})^{op})$ is finite if and only if $\fidim(A^{op})$ is finite.
\end{coro}

%

\begin{ej}
Let $\Lambda = \left( \begin{array}{cc} A& M \\ 0 & B \end{array}\right)$ where $A$ and $B$ are Artin algebras, $\gl (B) <\infty$ and $M$ is an $A$-$B$-bimodule. Then $\fidim(\Lambda) < \infty$ if and only if $\fidim(A) < \infty$.
\end{ej}

Compare the previous example with Theorem \ref{triangular}, where $C$ is the matrix algebra given by $A$ and $B$.  

\begin{prop}(Proposition 4.18 of \cite{GLP}) Let $A$ be an Artin algebra and $\mathfrak{I}$ be a strong idempotent ideal. Then
$$\fidim(A) \leq \pd_{A^{op}}(A/\mathfrak{I}) + \gl(A/\mathfrak{I})+\fidim (\Gamma).$$


\end{prop}

\subsection{Recollements and Igusa-Todorov $\phi$ function}

Following Example 1 of \cite{H}, Definition 2.1.1, and Remark 2.1.2 of \cite{CPS} we know that the prototypical example of a standard recollement is the one of the contexts of Section 7.1, originally given in \cite{APT}.\\

Let $\mathcal{T}_1$, $\mathcal{T}$ and $\mathcal{T}_2$ be triangulated categories. A recollement of $\mathcal{T}$ relative
to $\mathcal{T}_1$ and $\mathcal{T}_2$ is given by
$$\xymatrix{\mathcal{T}_1 \ar[rr]^{i_{\ast} = i_{!}} & & \mathcal{T} \ar[rr]^{j^{!} = j^{\ast}} \ar@<3ex>[ll]_{i^{!}} \ar@<-3ex>[ll]_{i^{\ast}} & & \mathcal{T}_2 \ar@<3ex>[ll]_{j_{\ast}} \ar@<-3ex>[ll]_{j_{!}}}$$

and denoted by $(\mathcal{T}_1 ,\mathcal{T} ,\mathcal{T}_2 ,i^{\ast} , i_{\ast} = i_{!}, i^{!}, j_{!}, j^{!} = j^{\ast} , j_{\ast})$ such that

\begin{itemize}

\item $(i^{\ast} , i_{\ast} )$, $(i_{!} , i^{!} )$, $(j_{!} , j^{!} )$ and $(j^{\ast}, j_{\ast} )$ are adjoint pairs of triangle functors,

\item $i_{\ast}$ , $j_{!}$ and $j_{\ast}$ are full embeddings,

\item $j^{!}i_{\ast} = 0$ (and thus also $i^{!} j_{\ast}  = 0$ and $i^{\ast} j_{!} = 0$),

\item for each $X \in \mathcal{T}$ , there are triangles

$$\xymatrix{j_! j^! X \ar[r] & X \ar[r]  & i_{\ast} i^{\ast} X \ar[r] & }$$
$$\xymatrix{i_! i^! X \ar[r] & X \ar[r]  & j_{\ast} j^{\ast} X \ar[r] & }$$

where the arrows to and from $X$ are the counits and the units of the adjoint
pairs respectively.

\end{itemize}

\begin{defi}

Let $A$, $B$ and $C$ be algebras. A recollement $(DB, DA, DC, i^{\ast} , i_{\ast} = i_! , i^! , j_! , j^! = j^{\ast} , j_{\ast} )$ is said to be standard and defined by $Y \in \mathcal{D}^b (A^{op} \otimes B)$ and $X \in \mathcal{D}^b (C^{op} \otimes A)$ if $i^{\ast} \cong \cdot \otimes^{L}_A Y$ and $j_{!} \cong \cdot \otimes^{L}_{C} X$.

\end{defi}

The result below is a generalization of the ones given in Section 7.1.

\begin{teo}(Theorem II of \cite{Q})
Let $A$, $B$ and $C$ be algebras and let  $(DB, DA, DC, i^{\ast} , i_{\ast} = i_! , i^! , j_! , j^! = j_{\ast} , j_{\ast} )$ be a standard recollement defined by $Y \in \mathcal{D}^b (A^{op} \otimes B)$ and $X \in \mathcal{D}^b (C^{op} \otimes A)$. Suppose $Y^{\ast} = R\hom_{B}(Y,B)$. Then the following holds true:

\begin{enumerate}

\item If $ _{C}X$ is compact, then $\fidim (C) \leq \fidim (A) + w(_CX)$.

\item If $Y^{\ast}_A$  is compact, then $\fidim(B) \leq \fidim(A) + w(Y_B)$.

\item If $_{A}Y$ and $Y^{\ast}_A$ are compact, then

$$\fidim (A) \leq \sup\{ \fidim (B) + \pd (_AY) + \pd (Y^{\ast}_A), \fidim(C)+w(X_A)\}.$$

\end{enumerate}

\end{teo}

\begin{obs}
If $A$ and $B$ are derived equivalent algebras, we can apply the previous theorem to the trivial recollement
$$\xymatrix{0 \ar[rr] & & \mathcal{D}(\mod A) \ar[rr]^{ \cdot \ \otimes^{L}_A P_B} \ar@<3ex>[ll] \ar@<-3ex>[ll] & & \mathcal{D}(\mod B) \ar@<3ex>[ll] \ar@<-3ex>[ll]_{\cdot \ \otimes^{L}_B P_A} }$$

and reobtain Theorem \ref{derivadamente eq}.
\end{obs}

\section{Algebras of $\Omega^{\infty}$-infinite representation type}

We recall the definition given at the beginning of Section 4. We say that $A$ is of $\Omega^n$-finite representation type if $\Omega^n(\mod A)$ is of finite representation type. If $A$ is not of $\Omega^n$-finite representation type for any $n \in \mathbb{N}$, we say that $A$ is of  $\Omega^{\infty}$-infinite representation type.
We begin this section with a result that allows us to compute the $\phi$-dimension of some special algebras.

\begin{teo}[Theorem 52 of \cite{BM}]\label{triangular}
Let $Q$ be a quiver and $C = \frac{\Bbbk Q}{I}$. Suppose there exist two disjoint full subquivers of $Q$, $\Gamma$ and $\bar{\Gamma}$ such that 

\begin{itemize}

\item  $Q_0 = \Gamma_0 \cup \bar{\Gamma}_0$, 

\item there is at least one arrow from $\Gamma_0$ to $\bar{\Gamma}_0$, and there are no arrows of $Q$ from $\bar{\Gamma}_0$ to $\Gamma_0$,

\item if $\alpha \in Q_1$ with $\st(\alpha) \in \Gamma_0$ and $\tg(\alpha) \in \bar{\Gamma}_0$ then $ \beta \alpha = 0 $ for every arrow $\beta$.

\end{itemize}

Then $\fidim (C) \leq \fidim (A) + \fidim (B) + 1$ where $A = \frac{\Bbbk \Gamma}{I \cap \Bbbk \Gamma}$ and $B = \frac{\Bbbk \bar{\Gamma}}{I \cap \Bbbk \bar{\Gamma}}$.

\end{teo}

There is a more general version of this result. See Theorem 3.3 of \cite{BM3}.\\

In \cite{M}, it is shown that some Gorenstein algebras are of $\Omega^{\infty}$-infinite representation type (see Example 4.1). However, these algebras have finite $\phi$-dimension since they are Gorenstein (see Theorem 4.7 of \cite{LM}). We now give an example of a non-Gorenstein algebra of $\Omega^{\infty}$-infinite representation type, but with finite $\phi$-dimension.  

\begin{ej}(Example 53 of \cite{BM})\label{finito}

Let $\Bbbk$ be an infinite field with characteristic different to $2$. Consider the algebra $A = \frac{\Bbbk Q}{I}$, where $Q$ is the quiver below:

$$ Q = \xymatrix{ 1 \ar@/_{1pc}/[r]^{\alpha_1} \ar@/_{2pc}/[r]^{\alpha_2} & 2 \ar@/_{1pc}/[l]^{\beta_1} \ar@/_{2pc}/[l]^{\beta_2} & 3 \ar[l]^{\gamma} \ar@(ur,dr)^{\delta}}$$

and $$I = \langle \beta_1\alpha_1-2\beta_2\alpha_2,\ \alpha_1\beta_1 - \alpha_2\beta_2, \ \delta^2 , \ \delta\gamma , \ \gamma\beta_1, \gamma\beta_2 , \ \alpha_i \beta_j, \ \beta_i \alpha_j \mbox{ with } \ i,j \in \{1,2\} \mbox{ and } i \not = j  \rangle.$$

It is left to the reader to compute that $A$ is of $\Omega^{\infty}$-infinite representation type.
It can be seen that the injective module $I(3)$ has infinite projective dimension, hence $A$ is not a Gorenstein algebra. 
Finally, by Theorem \ref{triangular} we conclude that $\fidim(A) = 1$.  

\end{ej}

In \cite{LM}, is proved that $\Omega^n$-finite representation type algebras has finite $\phi$-dimension. Therefore if we want to find an example of an algebra with infinite $\phi$-dimension, it must be of $\Omega^{\infty}$-infinite representation type.

The following example shows a radical cube zero algebra with infinite $\phi$-dimension. It is a counterexample of the conjecture stated in \cite{FLM}. However, its finitistic dimension is finite (see \cite{Z2}). Theorem \ref{derivadamente eq} implies that derived equivalences preserve infinite $\phi$-dimensions. Using this result, we can construct other examples with infinite $\phi$-dimension. The example also shows that the subcategory $\Omega^{\infty}(\mod A)$ can be trivial, although the algebra $A$ is of $\Omega^{\infty}$-infinite representation type.

\begin{ej}(Example 54 of \cite{BM})\label{infinito} Let $A = \frac{\Bbbk Q}{I}$ be an algebra where $Q$ is 

$$
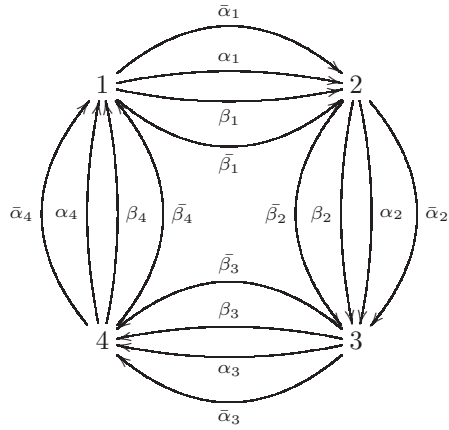
\xymatrix{ 1 \ar@/^8mm/[rrr]^{\bar{\alpha}_1} \ar@/^2mm/[rrr]^{\alpha_1} \ar@/_2mm/[rrr]_{\beta_1} \ar@/_8mm/[rrr]_{\bar{\beta_1}} & &  & 2 \ar@/^8mm/[ddd]^{\bar{\alpha}_2} \ar@/^2mm/[ddd]^{\alpha_2} \ar@/_2mm/[ddd]_{\beta_2} \ar@/_8mm/[ddd]_{\bar{\beta_2}} \\ & &  &\\& & & & \\ 4 \ar@/^8mm/[uuu]^{\bar{\alpha}_4} \ar@/^2mm/[uuu]^{\alpha_4} \ar@/_2mm/[uuu]_{\beta_4} \ar@/_8mm/[uuu]_{\bar{\beta_4}} &  & &  3 \ar@/^8mm/[lll]^{\bar{\alpha}_3} \ar@/^2mm/[lll]^{\alpha_3} \ar@/_2mm/[lll]_{\beta_3} \ar@/_8mm/[lll]_{\bar{\beta_3}}}$$
and $I = \langle \alpha_i\alpha_{i+1}-\bar{\alpha}_i \bar{\alpha}_{i+1},\ \beta_i \beta_{i+1}-\bar{\beta}_i \bar{\beta}_{i+1},\ \bar{\alpha}_i \alpha_{i+1},\ \alpha_i \bar{\alpha}_{i+1},\ \bar{\beta}_i \beta_{i+1},\ \beta_i \bar{\beta}_{i+1},\text{ for } i \in \mathbb{Z}_4,\ J^3 \rangle$.
It is clear that $\pd(S_i) = \infty$ for every $i \in \mathbb{Z}_4$.
Let $M^{\alpha_1}_n$ and $M^{\beta_1}_n$, with $n \in \mathbb{Z}^+$, be the indecomposable $A$-modules defined by

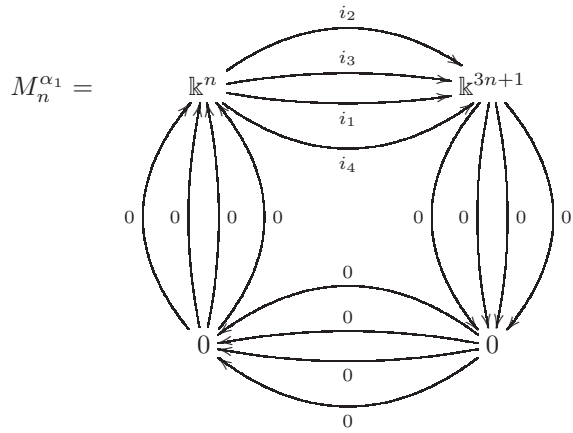
\[
M^{\alpha_1}_n =   \xymatrix{ & \Bbbk^n \ar@/^8mm/[rrr]^{i_2} \ar@/^2mm/[rrr]^{i_3} \ar@/_2mm/[rrr]_{i_1} \ar@/_8mm/[rrr]_{i_4} & &  & \Bbbk^{3n+1} \ar@/^8mm/[ddd]^{0} \ar@/^2mm/[ddd]^{0} \ar@/_2mm/[ddd]_{0} \ar@/_8mm/[ddd]_{0} \\ & & &  & \\ & & & & & \\ & 0 \ar@/^8mm/[uuu]^{0} \ar@/^2mm/[uuu]^{0} \ar@/_2mm/[uuu]_{0} \ar@/_8mm/[uuu]_{0} &  & &  0 \ar@/^8mm/[lll]^{0} \ar@/^2mm/[lll]^{0} \ar@/_2mm/[lll]_{0} \ar@/_8mm/[lll]_0}
\]

\[
 M^{\beta_1}_n = 
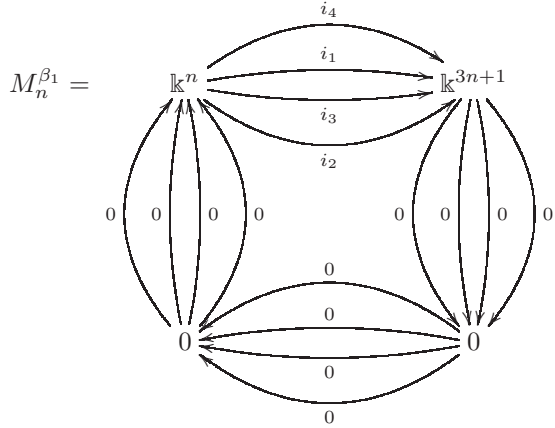
\xymatrix{ \Bbbk^n \ar@/^8mm/[rrr]^{i_4} \ar@/^2mm/[rrr]^{i_1} \ar@/_2mm/[rrr]_{i_3} \ar@/_8mm/[rrr]_{i_2} & &  & \Bbbk^{3n+1} \ar@/^8mm/[ddd]^{0} \ar@/^2mm/[ddd]^{0} \ar@/_2mm/[ddd]_{0} \ar@/_8mm/[ddd]_{0} \\ & &  &\\& & & & \\ 0 \ar@/^8mm/[uuu]^{0} \ar@/^2mm/[uuu]^{0} \ar@/_2mm/[uuu]_{0} \ar@/_8mm/[uuu]_{0} &  & &  0 \ar@/^8mm/[lll]^{0} \ar@/^2mm/[lll]^{0} \ar@/_2mm/[lll]_{0} \ar@/_8mm/[lll]_0}
\]
where the linear transformations $i_m: \Bbbk^n \rightarrow \Bbbk^{3n+1}$, with $m \in \{1,2,3,4\}$, verifies:

\begin{itemize}

\item  $i_1(e_j) = f_j\ \forall j \in \{1, \ldots n\}$.

\item  $i_2(e_j) =f_{n+j}\ \forall j \in \{1, \ldots n\}$.

\item  $i_3(e_j) =f_{n+j+1}\ \forall j \in \{1, \ldots n\}$.

\item  $i_4(e_j) = f_{2n+j + 1}\ \forall j \in \{1, \ldots n\}$.

\end{itemize}

where $\{e_1 \ldots e_n\}$ and $\{f_1,\ldots, f_{3n+1}\}$ are the canonical bases of $\Bbbk^n$ and $\Bbbk^{3n+1}$ respectively.

In an analogous way we define $M^{\alpha_2}_n, M^{\alpha_3}_n, M^{\alpha_4}_n$ and $M^{\beta_2}_n, M^{\beta_3}_n, M^{\beta_4}_n$. 
%
%
%
%
%
%
%
%
%
%

We have $\phi(M^{\alpha_i}_n \oplus M^{\beta_i}_n) = n-1$ for any $n \geq 2$, then we conclude that $\fidim(A) = \infty$.

\end{ej}

\section{The symmetry problem}

\noindent We recall from \ref{desigualdades}, the homological inequalities:

\begin{equation}
\fin (A) \leq \fidim (A) \leq \psidim (A) \leq \gl (A) 
\end{equation}

Homological dimensions can be defined with reference to either left or right modules. Sometimes, these dimensions coincide. That is, the homological measure is left-right symmetric. This is the case with the global dimension of a noetherian ring (see \cite{A}). Also, the dominant dimension of a finite dimensional algebra is symmetric (see \cite{Mu}). But, it is not true for the finitistic dimension (see \cite{GKK}, \cite{JL}). In fact, in a recent paper, Cummings proves that (Theorem 3.4 of \cite{C})``the finitistic dimension conjecture holds for all finite dimensional algebras if and only if, for all finite dimensional algebras, the finitistic dimension of an algebra being finite implies that the finitistic dimension of its opposite algebra is also finite". So, it is natural to try to give an answer to the following question. When do the right and the left Igusa\ -Todorov dimensions are equal for an Artin algebra? That is, for a given Artin algebra $A$: 
$$\fidim (A) \overset{?}{=} \fidim (A^{op})$$
$$\psidim (A) \overset{?}{=} \psidim (A^{op})$$

\noindent The next result shows a large family of algebras where their left and right dimensions coincide.

\begin{coro}(Corollary 4.9 of \cite{LM})
If $A$ is a  Gorenstein algebra, then: 
$$\fidim (A) = \fidim (A^{op}) = \psidim (A) = \psidim (A^{op}).$$ 
\end{coro}

\noindent But in \cite{M*} was exhibited the example below, where the left and the right Igusa-Todorov $\psi$-dimensions are different.

\begin{ej} Consider the algebra $A = \frac{\Bbbk Q}{J^2}$, where $Q$ is the following quiver: \vspace{.1cm}
$$Q = \hspace{.5cm} \xymatrix{ 1 \ar@(l,u) \ar[r]& 2 \ar[r] & 3 \ar[r] & \cdots \ar[r] &  n}.$$

\noindent Therefore $\psi \dim (A) = 2n-3$ and $\psi \dim (A^{op}) = \fidim (A^{op}) = \fidim (A) = n-1$. Then for $n>2$, $\psi \dim (A^{op}) \neq \psi \dim (A)$.

\end{ej}

\noindent Regarding the Igusa-Todorov $\phi$-function, was proved in \cite{LMM}, \cite{BMR}, the following results.

\begin{teo}(Theorem 4.21 of \cite{LMM})\label{Rad^2izq=der}
If $A$ is a radical square zero algebra then $\fidim (A) = \fidim(A^{op})$.\\
\end{teo}

\noindent In \cite{BMR} was given a generalization of the previous result.

\begin{teo} (Theorem 3 of \cite{BMR}) \label{Truncada2izq=der}
Let $A = \frac{\Bbbk Q}{J^k}$ and $A^{op}= \frac{\Bbbk Q^{op}}{(J^{op})^k}$ be truncated path algebras, then
$$\fidim (A) = \fidim (A^{op}).$$
\end{teo}
\vspace{0.3cm}

However, the following example from \cite{BM3} shows that the Igusa-Todorov $\phi$-function is not symmetric in general.

\begin{ej}(\cite{BM3}) Consider the finite dimensional algebra $C_{p,q} =\frac{\Bbbk Q}{I_{p,q}}$ where $Q$ is the following quiver

$$
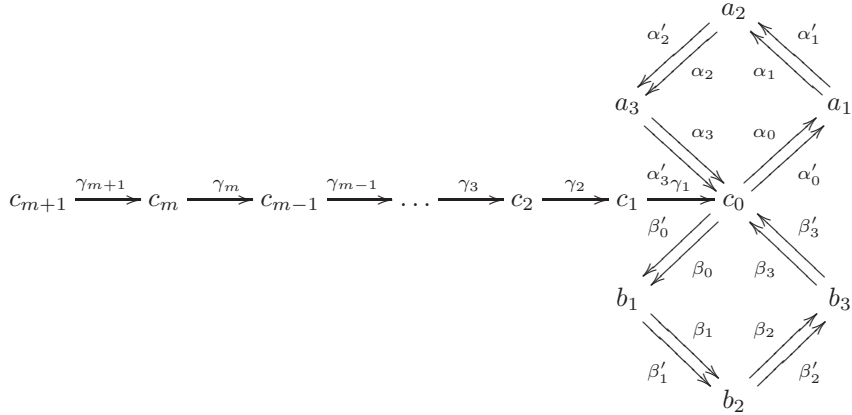
\xymatrix{ &&&&&& a_2 \ar@<.5ex>[dl]^{\alpha_2} \ar@<-.5ex>[dl]_{\alpha'_2} && \\&&&&& a_3 \ar@<.5ex>[dr]^{\alpha_3} \ar@<-.5ex>[dr]_{\alpha'_3} && a_1 \ar@<.5ex>[ul]^{\alpha_1} \ar@<-.5ex>[ul]_{\alpha'_1}& \\ c_{m+1} \ar[r]^{\gamma_{m+1}} &  c_m \ar[r]^{\gamma_m}&c_{m-1}\ar[r]^{\gamma_{m-1}}& \ldots \ar[r]^{\gamma_3}& c_2 \ar[r]^{\gamma_2}& c_1 \ar[r]^{\gamma_1}& c_0\ar@<.5ex>[ur]^{\alpha_0} \ar@<-.5ex>[ur]_{\alpha'_0} \ar@<.5ex>[dl]^{\beta_0} \ar@<-.5ex>[dl]_{\beta'_0} &&  \\&&&&& b_1 \ar@<.5ex>[dr]^{\beta_1} \ar@<-.5ex>[dr]_{\beta'_1} && b_3 \ar@<.5ex>[ul]^{\beta_3} \ar@<-.5ex>[ul]_{\beta'_3}& \\&&&&&& b_2 \ar@<.5ex>[ur]^{\beta_2} \ar@<-.5ex>[ur]_{\beta'_2} && }$$


and $I_{p,q}$ is generated by

\begin{itemize}
\item $\gamma_{i+1}\gamma_{i}\ \forall i=1,\ldots,m$,

\item $\alpha_{i}\alpha_{i+1}, \beta_{i}\beta_{i+1}, \alpha'_{i}\alpha'_{i+1}, \beta'_{i}\beta'_{i+1}\ \forall i \in \mathbb{Z}_4$,

\item $\alpha'_{i}\alpha_{i+1} - \alpha_{i}\alpha'_{i+1}, \beta'_{i}\beta_{i+1}-\beta_{i}\beta'_{i+1}\ \forall i \in \mathbb{Z}_4$,

\item $ \alpha_3\beta_0 , \alpha'_3\beta'_0, \beta_3\alpha_0, \beta'_3\alpha'_0$ and
\item $\alpha'_3\beta_0 -\alpha_3\beta'_0, \beta'_3\alpha_0-p\beta_3\alpha'_0, \gamma_1\alpha_0-q\gamma_1\alpha'_0, \gamma_1\beta'_0-\gamma_1\beta_0$ with $\mathbb{Q}[p,q] \cong \mathbb{Q}[x,y]$. 
\end{itemize}

Then $\phi \dim (C_{p,q}) = 5$ and $\fidim(C_{p,q}^{op}) \geq m-1$.

\end{ej}

\section{Generalizations}

This section presents generalizations and adaptations of the Igusa-Todorov functions to other contexts. In each subsection, the original notation used by the respective author was respected.

\subsection{Igusa-Todorov functions on comodules over semiperfect coalgebras}\ 

In this subsection, we recall the main result of \cite{HLM17}, where the authors adapted the Igusa-Todorov functions for right comodules over left semiperfect coalgebras. In particular, we recall that Theorem \ref{qcF}, which characterizes the quasi-co-Frobenius coalgebras with the Igusa-Todorov functions for comodules, is an analogous to Theorem \ref{selfinjective} (Theorem 5 of \cite{HL}), which classify selfinjective algebras via Igusa-Todorov functions.
Assume $C$ is a coalgebra over a field $\Bbbk$. In that case, we denote by $\mathcal {M}^C$ and $^C\!\!\mathcal {M}$ the categories of right and left comodules over $C$ respectively and by $\mathcal {M}^C_f$ and $^C\!\!\mathcal {M}_f$ the respective complete subcategories of finite dimensional comodules.

\begin{defi}  A coalgebra $C$ is 
\begin{itemize} 

\item  left (right) semiperfect if all indecomposable injective right (left) comodules are finite dimensional,

\item  left (right) quasi-co-Frobenius, shortly  left qcF, if every injective right (left) C-comodule
is projective.

\end{itemize}
\end{defi}

Let $C$ be a left semiperfect coalgebra and $K(C)$ be the free abelian group
generated by all symbols $[M]$ with $M \in \ \mathcal {M}^C_f$ under
the relations

\begin{enumerate}
  \item $[M]-[M']-[M'']$ if $M\cong M' \oplus M''$,
  \item $[I]$ if $I$ is injective.
\end{enumerate}

\noindent Then $K(C)$ is the free abelian group generated by all
isomorphism classes of indecomposable noninjective objects in
$\mathcal {M}^C_f$. Note that if $C$ is left semiperfect, the injective envelope of a finite dimensional comodule is also finite dimensional, so the cosyzygy of such a comodule will also be finite dimensional. Moreover, as $\Omega^{-1}$ respects direct sums
and sends injective comodules to 0, it gives rise to a group
homomorphism (which we also call $\Omega^{-1}$) $\Omega^{-1}: K(C) \rightarrow K(C)$.

For any $M \in \ \mathcal {M}^C_f$, let $\langle M \rangle$ denote the subgroup of $K(C)$ generated by all the symbols $[N]$, where $N$ is an indecomposable noninjective direct summand of $M$. Since the rank of $\Omega^{-1}(\langle M \rangle)$ is less or equal to the rank of $\langle M \rangle$, which is finite, there exists a non-negative integer $n$ such that the rank of $\Omega^{-n}(\langle M\rangle)$ is equal to the rank of $\Omega^{-i}(\langle M \rangle)$ for all $i \geq n$. Let $\varphi(M)$ denote the least such $n\in \mathbb N$.

\begin{defi}\label{ITcoalgebras}
The function defined in this way $\varphi: \mathcal {M}^C_f \to \mathbb N$ is called the Igusa-Todorov function (IT-function) for the left semiperfect coalgebra $C$, and its $\varphi$-dimension is $\varphi \dim (C) = \sup\left \{\varphi(M) \mbox{ with }M \in \mathcal {M}^C_f\right \}$. 
\end{defi}

The following is the main theorem of \cite{HLM17}. This result characterizes the qcF coalgebras via the Igusa-Todorov function, defined above.

\begin{teo}(Theorem 3.4 of \cite{HLM17})\label{qcF} A coalgebra $C$ is left qcF if and only if it is left semiperfect and $\varphi \dim (C) = 0$.

\end{teo}

\subsection{Igusa-Todorov functions on derived categories over Artin algebras}\ 

In this subsection, we recall Denming Xu's article \cite{Xu}, which generalizes the Igusa-Todorov functions for derived categories.
Through this subsection $A$, denotes a finite dimensional algebra.
We recall that $\mathcal{D}^b(\mod A)$ is a Krull-Schmidt category \cite{BS}, a central key for lifting the Igusa-Todorov functions to derived categories.

\begin{defi}(\cite{Xu})
Let $K_0 (\mathcal{RK}^{\leq 0}(\mathcal{P}_A))$ be the abelian group generated by all symbols $[P^{\bullet}]$, where $P^{\bullet} \in$ 
$\mathcal{RK}^{\leq 0}(\mathcal{P}_A)$, modulo the relations

\begin{itemize}

\item $[P^{\bullet}_1]+[P^{\bullet}_2] = [P^{\bullet}]$ if $P^{\bullet} \simeq P_1^{\bullet}\oplus P_2^{\bullet}$

\item $[P^{\bullet}] = 0$ if $P^{\bullet}$ is a stalk complex concentrated in degree zero.

\end{itemize}

\end{defi}

$K_0 (\mathcal{RK}^{\leq 0}(\mathcal{P}_A))$ is the free abelian group generated by the isomorphism classes of
all indecomposable objects in $\mathcal{RK}^{\leq 0}(\mathcal{P}_A)$ except stalk complexes concentrated in degree zero.
For any complex $P^{\bullet} \in \mathcal{RK}^{\leq 0}(\mathcal{P}_A)$ we define $\Omega (P^{\bullet}) = \tau_{\leq -1}(P^{\bullet}) [-1]$. Note that $\Omega$ commutes with direct sums and $\Omega(P^{\bullet}) = 0$ if $P^{\bullet}$ is a stalk complex concentrated in degree zero. It induces a homomorphism $\bar{\Omega}: K_0 (\mathcal{RK}^{\leq 0}(\mathcal{P}_A)) \rightarrow K_0 (\mathcal{RK}^{\leq 0}(\mathcal{P}_A))$. Consider $P^{\bullet} \in \mathcal{RK}^{\leq 0}(\mathcal{P}_A))$, then we denote by $\langle \add P^{\bullet} \rangle$ the subgroup of $ K_0 (\mathcal{RK}^{\leq 0}(\mathcal{P}_A))$  generated by the isomorphism classes of all indecomposable direct summands of $P^{\bullet}$. By Lemma \ref{Fitting} we define 
$$\phi(P^{\bullet}) = \eta_{\bar{\Omega}}(\langle \add P^{\bullet} \rangle) $$
and
$$\psi (P^{\bullet}) = \phi(P^{\bullet}) + \sup\{\pd S^{\bullet}\text{: } \pd S^{\bullet}<\infty \text{ and } S^{\bullet} \text{ is a direct summand of } \bar{\Omega}^{\phi(P^{\bullet})}({P^{\bullet}}) \}$$

Finally, for any $X^{\bullet} \in \mathcal{D}^{b}(\mod A)$ such that $H^{j}(X^{\bullet}) = 0$ for all $j \geq 1$ there exists a unique $P^{\bullet}_{X^{\bullet}} \in \mathcal{RK}^{\leq 0}(\mathcal{P}_A)$ such that $P^{\bullet}_{X^{\bullet}}$ 
is quasi-isomorphic to  $X^{\bullet}$. So we define

$$\phi(X^{\bullet}) = \phi(P^{\bullet}_{X^{\bullet}} ) \text{ and } \psi(X^{\bullet}) = \psi(P^{\bullet}_{X^{\bullet}} )$$
 
Note that if $X^{\bullet}$ is a stalk complex concentrated in degree zero, then $\phi(X^{\bullet}) = \phi (X^0)$ and $\psi(X^{\bullet}) = \psi (X^0)$ where the second terms are the original functions defined by K. Igusa and G. Todorov, \cite{IT}.\\ 

The remark below lists the general version for derived categories of all the usual properties of the original Igusa-Todorov functions. See Proposition \ref{it1}. 
 
\begin{obs}(Lemma 2.2 and Lemma 2.3 of \cite{Xu}) Let $P^{\bullet}$ and $Q^{\bullet}$ be in $\mathcal{RK}^{\leq 0}(\mathcal{P}_A)$. Then we have the following.

\begin{enumerate}

\item If $\pd P^{\bullet}$ is finite, then $\phi (P^{\bullet}) = \psi(P^{\bullet} ) = \pd P^{\bullet}$.

\item If $P^{\bullet}$ is indecomposable and $\pd P^{\bullet}$ is infinite, then $\phi (P^{\bullet}) = \psi(P^{\bullet} )= 0$.

\item $\phi ( P^{\bullet} ) \leq \phi ( P^{\bullet} \oplus  Q^{\bullet} )$ and $\psi ( P^{\bullet} ) \leq \psi ( P^{\bullet} \oplus  Q^{\bullet} )$.

\item $\phi ( k P^{\bullet} ) = \phi ( P^{\bullet} )$ and $\psi ( k P^{\bullet} ) = \psi ( P^{\bullet} )$. 

\item Let $T^{\bullet}$ be a direct summand of $\Omega^n (  P^{\bullet} )$ with $n \leq \phi( P^{\bullet} )$. If $\pd T^{\bullet}  < \infty$, then $\pd T^{\bullet} + n \leq \psi( P^{\bullet} )$.

\end{enumerate}

\end{obs}

The following remark generalizes the item 5 of Proposition \ref{it1}. The proof given in \cite{HLM}, can be easily adapted.

\begin{obs}

If $P^{\bullet} \in \mathcal{RK}^{\leq 0}(\mathcal{P}_A)$, then:
\begin{enumerate}

\item $\phi(P^{\bullet}) \leq \phi(\Omega(P^{\bullet}))+1$.

\item $\psi(P^{\bullet}) \leq \psi(\Omega(P^{\bullet}))+1$.
\end{enumerate}
\end{obs}

Finally, in this subsection, we recall the generalization of the inequality given in the main result of the famous article \cite{IT}.

\begin{teo} (Theorem 3.5 of \cite{Xu}) Let $\xymatrix{X^{\bullet} \ar[r]& Y^{\bullet} \ar[r] & Z^{\bullet} \ar[r] & X^{\bullet}[1]}$ be a distinguished triangle in $D^b(\mod A)$ such that $H_j (X^{\bullet}) = H_j (Y^{\bullet}) = 0$ for all $j \geq 1$. Suppose $\pd Z^{\bullet}  < \infty$, then
$$\pd Z^{\bullet}  \leq \psi (X^{\bullet}  \oplus Y^{\bullet} ) + 1.$$

\end{teo}

\subsection{Igusa-Todorov $\Phi_{[\mathcal{D}]}$ and $\Psi_{[\mathcal{D}]}$ functions and LIT-algebras}\

We start this section by recalling a new version of Fitting's Lemma.

\begin{lema}(Lemma 3.3 of \cite{BLMV}){\label{lemadegrupos}} Let $G$ be a free abelian group, $D$ be a subgroup of $G$, $L \in End_ Z (G)$
be such that $L(D) \subset D$ and let $k$ be a positive integer for which $L : L^k (D) \rightarrow D$ is
a monomorphism. Then, for each finitely generated subgroup $X \subset G$, we have that

$$ \eta_L (X) \leq \eta_{\bar{L}} (\bar{X}) + k,$$

where $\bar{L} : G/D \rightarrow G/D$, $g + D  \rightarrow L(g) + D$, and $\bar{X} = (X + D)/D$.

\end{lema}

Fix a subclass $\mathcal{D} \subset \mod A$ that is additively closed: $\mathcal{D} = \add \mathcal{D}$, and invariant by $\Omega$: $\Omega (\mathcal{D}) \subset \mathcal{D}$. Then the quotient group $\mathcal{K}_{\mathcal{D}} = \frac{K_0}{\langle \mathcal{D} \rangle }$ is a free abelian group.

\begin{defi}(\cite{BLMV})
Let $A$ be an Artin algebra and $\mathcal{D} \subset \mod A$ be a class of A-modules additively closed and invariant by $\Omega$. Let $ \bar{\Omega}_{\mathcal{D}}: K_{\mathcal{D}} \rightarrow K_{\mathcal{D}}$ be the endomorphism defined by $ \bar{\Omega}_{\mathcal{D}}([X]+\langle \mathcal{D} \rangle) = [\Omega (X)] + \langle \mathcal{D} \rangle$.  For any
$X \in \mod A$, we set

$$\Phi_{[\mathcal{D}]}(X) = \eta_{\bar{\Omega}_{\mathcal{D}}} (\overline{ \langle X \rangle }) $$

where $\overline{ \langle X \rangle } =(\langle X \rangle + \langle \mathcal{D} \rangle )/\langle \mathcal{D} \rangle $. 

\end{defi}

Note that $\overline{ \langle X \rangle }$ is the free abelian group generated by
the set of all the iso-classes of indecomposable nonprojective $A$-modules of $\add X$
that do not lie in $\mathcal{D}$. From Lemma \ref{lemadegrupos}, the authors obtained the following result, which compares the value of the first Igusa-Todorov function with its generalized version.

\begin{teo}(Theorem 3.5 of \cite{BLMV}) \label{fisubD} Let $A$ be an Artin algebra and $\mathcal{D} \subset \mod A$ be a class of $A$-modules additively closed and invariant by $\Omega$. Then, 

$$\Phi(X) \leq \Phi_{[\mathcal{D}]} (X) + \fidim(\mathcal{D}),$$ 

for every $X \in \mod A$.

\end{teo}

\begin{defi}(\cite{BLMV})
Let $A$ be an Artin algebra and $\mathcal{D} \subset \mod A$ be a class of $A$-modules additively closed and invariant by $\Omega$. For any $X \in \mod A$, we set

$$ \Psi_{[\mathcal{D}]} (X) = \Phi_{[\mathcal{D}]} (X) + \fin (\{Z \in \mod A : Z \mid \Omega^{\Phi_{[\mathcal{D}]} (X)} (X)\}).$$

\end{defi}

\begin{defi}(\cite{BLMV})
Let $A$ be an Artin algebra and $\mathcal{D} \subset \mod A$ be a class of $A$-modules  additively closed and invariant by $\Omega$, for $\mathcal{X} \subset \mod A$ we define:

\begin{itemize}

\item $\Phi_{[\mathcal{D}]} \dim (\mathcal{X}) = \sup \{ \Phi_{[\mathcal{D}]}(X): X \in \mathcal{X}\}$

\item $\Psi_{[\mathcal{D}]} \dim (\mathcal{X}) = \sup \{ \Psi_{[\mathcal{D}]}(X): X \in \mathcal{X}\}$

\end{itemize}

In particular, we denote by
$$\Phi_{[\mathcal{D}]}\dim(A) = \Phi_{[\mathcal{D}]}\dim(\mod A),\text{and } \Psi_{[\mathcal{D}]}\dim(A) = \Psi_{[\mathcal{D}]}\dim(\mod A),$$
also called the $\Phi_{[\mathcal{D}]}$-dimension and $\Psi_{[\mathcal{D}]}$-dimension of $A$ respectively. 
\end{defi}

The next remark summarizes the properties of the extended versions of the Igusa-Todorov functions.

\begin{obs}(Propositions 3.9, 3.10, and 3.12 of \cite{BLMV}) \label{propiedades fisubD} Let $A$ be an Artin algebra and $\mathcal{D} \subset \mod A$ be a class of $A$-modules additively closed and invariant by $\Omega$. Then, the following statements hold true for $X, Y, M \in \mod A$.

\begin{enumerate}

\item If $M \in \mathcal{D} \cup \mathcal{P}(A)$, then $\Phi_{[\mathcal{D}]} (M ) = 0$ and $\Phi_{[\mathcal{D}]} (X \oplus M ) = \Phi_{[\mathcal{D}]} (X)$.

\item $\Phi_{[\mathcal{D}]} (X) \leq  \Phi_{[\mathcal{D}]} (X \oplus Y )$ and $\Psi_{[\mathcal{D}]} (X) \leq \Psi_{[\mathcal{D}]} (X \oplus Y )$.

\item $\Phi_{[\mathcal{D}]} \dim (\add X)= \Phi_{[\mathcal{D}]}(X)$ and $\Psi_{\mathcal{[D]}} \dim(\add X) = \Psi_{[\mathcal{D}]} (X)$.

\item $\Phi_{\mathcal{[D]}} (M) \leq\Phi_{\mathcal{[D]}} (\Omega(M))+1$ and $\Psi_{\mathcal{[D]}} (M) \leq\Psi_{\mathcal{[D]}} (\Omega(M))+1$.

\item If $Z$ is a direct summand of $\Omega^{n}(X)$ where $0 \leq t\leq\Phi_{[\mathcal{D}]}(X)$ and
$\pd(Z) < \infty$, then $\pd(Z) + t \leq\Psi_{[\mathcal{D}]}(X)$.

\item Suppose that $\fidim(\mathcal{D}) = 0$.

\begin{enumerate}

\item If $\pd (X) < \infty$, then $\Phi_{[\mathcal{D}]} (X) = \Phi(X) = \pd (X)$.

\item $\Psi(X) \leq \Psi_{[\mathcal{D}]} (X)$

\item If $M \in \mathcal{D} \cup \mathcal{P}(A)$, then $\Psi_{[\mathcal{D}]} (X \oplus M ) = \Psi_{[\mathcal{D}]} (X)$.

\item $\Psi_{[\mathcal{D}]} \dim(\mathcal{D}) = 0$

\end{enumerate}
\end{enumerate}

\end{obs}

The Theorem \ref{fisubD} and Item 6.(b), of the previous remark show how to compare the original Igusa-Todorov functions with the generalized versions.\\

We now show that $\Phi_{[\mathcal{D}]}$-dimension and $\Phi$-dimension can be equal for some families of modules $\mathcal{D}$.

\begin{defi} (\cite{BLMV})
Let $A$ be an Artin algebra. 
\begin{enumerate} 
\item A subclass $\mathcal{D} \subset \mod A$ is a {\bf $0$-Igusa-Todorov subcategory} if: $\mathcal{D} = \add (\mathcal{D})$, $\Omega (\mathcal{D}) \subset \mathcal{D}$ and $\fidim(\mathcal{D}) = 0$.
\item 
If the $0$-Igusa-Todorov subcategory $\mathcal{D} \subset \mod A$ verifies
for every $D \in \mathcal{D}$ that there are $A$-modules $D_1, D_2 \in \mathcal{D}$ and $m \in \mathbb{Z}^{+}$ such that $\ \bar{\Omega}([D_1] - [D_2]) = m[D]$,  we call it a {\bf $0$-Igusa-Todorov subcategory with preimages}.  
\end{enumerate}
\end{defi}

\begin{obs} The following families of modules are examples of $0$-Igusa-Todorov subcategories with preimages:
\begin{itemize} 
\item $\mathcal{GP}_A$ is a $0$-Igusa-Todorov subcategory with preimages.

\item If $\mathcal{D} \subset \mathcal{GP}_A$ and for every $D \in \mathcal{D}$ there is a $D' \in \mathcal{D}$ such that $\Omega(D') = D$, then $\mathcal{D}$ is a  $0$-Igusa-Todorov subcategory with preimages.

\item Let $M \in \mod A$, such that $\mathcal{D} = \add M$ is invariant by $\Omega$ and $\phi(M) = 0$. Then $\mathcal{D}$ is a $0$-Igusa-Todorov subcategory with preimages: it is not hard to see that, if there exists $M_i$ indecomposable summand of $M$ without preimage because $\mathcal{D}$ is of finite representation type, then $\fidim(\mathcal{D})$ is not zero.

\end{itemize}
 
\end{obs}

The example below shows a $0$-Igusa-Todorov subcategory that is not a $0$-Igusa-Todorov subcategory with preimages. 

\begin{ej} \label{ej sizigia-infinita}
Let $\Bbbk$ be an infinite field with characteristic different to $2$. Consider the algebra $A = \frac{\Bbbk Q}{I}$, where $Q$ is the quiver below:

$$ Q = \xymatrix{ 1 \ar@/_{1pc}/[r]^{\alpha_1} \ar@/_{2pc}/[r]^{\alpha_2} & 2 \ar@/_{1pc}/[l]^{\beta_1} \ar@/_{2pc}/[l]^{\beta_2} }$$

and $I = \langle \beta_1\alpha_1-2\beta_2\alpha_2, \alpha_1\beta_1 - \alpha_2\beta_2,  \alpha_i \beta_j , \beta_i \alpha_j \mbox{ with } i,j \in \{1,2\}, i\not = j\rangle$.\\

The following assertions are clear.
\begin{itemize}

\item $A$ is selfinjective.

\item $ M_a = \xymatrix{ \Bbbk \ar@/_{1pc}/[r]^{a1_{\Bbbk}} \ar@/_{2pc}/[r]^{1_{\Bbbk}} & \Bbbk \ar@/_{1pc}/[l]^{0} \ar@/_{2pc}/[l]^{0} }$ are indecomposable modules for $a \neq 0$. 

\item $ N_a = \xymatrix{ \Bbbk \ar@/_{1pc}/[r]^{0} \ar@/_{2pc}/[r]^{0} & \Bbbk \ar@/_{1pc}/[l]^{1_{\Bbbk}} \ar@/_{2pc}/[l]^{a1_{\Bbbk}} }$ are indecomposable modules for $a \neq 0$.

\item $M_a \ncong M_b$ and $N_a \ncong N_b$ if $a \neq b$.

\item $\Omega (M_a) = N_{-\frac{a}{2}}$ and $\Omega (N_{a}) = M_{-a}$ for $a \neq 0$.

\end{itemize}

If we consider $\mathcal{D}_{0} = \add \{M_{2^{-n}}, N_{-2^{-n}} \forall n \in \mathbb{N}\}$, then $\Omega(\mathcal{D}) \subset \mathcal{D}$. However there is no module $M$ in $\mathcal{D}_{0}$ such that $\Omega(M) = M_1$. 
\end{ej}

\begin{obs} \label{remarquita} Let $\{M_1 \ldots M_m\}$ a set of $A$-modules, $\mathcal{D}$ an additively closed subcategory and $I \cup J = \{1,\ldots, m\}$ a disjoint union, then if $k (\sum_{j \in J} \alpha_j[M_j] - \sum_{i \in I} \alpha_i[M_i])  \in \langle \mathcal{D} \rangle$ for $k \in \mathbb{Z}^{+}$ then $\sum_{j \in J} \alpha_j[M_j] - \sum_{i \in I} \alpha_i[M_i] \in \langle \mathcal{D} \rangle$.
\end{obs}

\begin{teo}\label{comparar en gral}
Let $A$ be an Artin algebra and $\mathcal{D}$ a $0$-Igusa-Todorov subcategory with preimages. If $M$ is an $A$-module, then there is an $A$-module $D \in \mathcal{D}$ such that $\phi(M\oplus D) = \Phi_{[\mathcal{D}]}(M)$.
\begin{proof}

It is clear that  $\phi (M\oplus D) \leq \Phi_{[\mathcal{D}]}(M\oplus D) = \Phi_{[\mathcal{D}]} (M) $ for every $A$-module $D \in \mathcal{D}$ by Theorem \ref{fisubD}.
Let $M = \oplus_{i = i}^m M_i$ be the decomposition into indecomposable modules of $M$. We can assume that $M_i \not \cong M_j$ for $i \not = j$ by Remark \ref{propiedades fisubD}. If $\phi_{[\mathcal{D}]}(M) = n$, then there are two disjoint sets $I, J$, two modules $D_1, D_2 \in \mathcal{D}$ and $\alpha_i \in \mathbb{N}$ $\forall i =1,\ldots m$ such that $I \cup J = \{1,\ldots, m\}$ is a disjoint union,
$$\sum_{j \in J} \alpha_j[\Omega^{n-1}(M_j)]+[\bar{D}_1] \not = \sum_{i \in I} \alpha_i[\Omega^{n-1}(M_i)]+[\bar{D}_2]$$
for every $\bar{D}_{i} \in \mathcal{D}$, for $i=1,2$, and
$$\sum_{j \in J} \alpha_j[\Omega^{n}(M_j)]+[D_1] = \sum_{i \in I} \alpha_i[\Omega^{n}(M_i)]+[D_2].$$
with $D_i \in \mathcal{D}$ for $i = 1,2$.
Consider $\tilde{D}_i$ and $\tilde{D}'_i$ such that $\bar{\Omega}^n([\tilde{D}_i]-[\tilde{D}'_i]) = l[D_i]$ for $i =1,2$. It follows, from the previous assertion and Remark \ref{remarquita}, that 
$$\phi(M \oplus \tilde{D}_1 \oplus \tilde{D}'_1 \oplus \tilde{D}_2 \oplus \tilde{D}'_2) \geq n = \Phi_{[\mathcal{D}]}(M).$$
\end{proof}
\end{teo}

\begin{coro}\label{igualdad} Let $A$ be an Artin algebra and $\mathcal{D}$ a $0$-Igusa-Todorov subcategory with preimages, then
$$\fidim (A) = \Phi_{[\mathcal{D}]}\dim (A)$$

\end{coro}

\begin{coro} \label{desigualdad}Let $A$ be an Artin algebra. If $\mathcal{D}_1 \subset \mathcal{D}_2 \subset \mod A$ where $\mathcal{D}_2$ is a $0$-Igusa-Todorov subcategory, and $\mathcal{D}_1$ is a $0$-Igusa-Todorov subcategory with preimages. Then for every $A$-module $M$

$$\Phi_{[\mathcal{D}_1]}(M) \leq \Phi_{[\mathcal{D}_2]} (M).$$

\begin{proof}

Since $\mathcal{D}_1$ is a $0$-Igusa-Todorov subcategory with preimages, by Theorem \ref{comparar en gral} there is a module $D \in \mathcal{D}_1$ such that
$$\phi(M\oplus D) = \Phi_{[\mathcal{D}_1]}(M)$$
On the other hand, by Theorem \ref{fisubD} and Remark \ref{propiedades fisubD} item $1$, we have the following inequality
$$\phi(M\oplus D) \leq \Phi_{[\mathcal{D}_2]}(M)$$
because $D \in \mathcal{D}_1 \subset \mathcal{D}_2$, and the thesis is obtained. 
\end{proof}

\end{coro}

The following example shows a module where the inequality in Corollary \ref{desigualdad} can be strict for a $0$-Igusa-Todorov subcategory $\mathcal{D}_1$ without preimages.

\begin{ej} Consider the algebra $A = \frac{\Bbbk Q}{I}$, where

$$Q = \xymatrix{ 1 \ar@/^/[r]^{\alpha} & 2 \ar@/^/[l]^{\beta} \ar@(ur,dr)^{\gamma} } \mbox{ and } I = \langle \beta \alpha, \gamma^2  \rangle$$

Note that the module $\gamma  A$ verifies $\Omega (\gamma A ) \cong \gamma A$, and it is the only indecomposable Gorenstein projective module that is not projective by Theorem 4.11 of \cite{ChShZh}.

It is easy to see that $\Omega(S_2) = \gamma A \oplus S_1$, then $\phi(S_2) = 0$. On the other hand $\pd(S_1) = 1$, then $\Phi_{[\mathcal{GP}_A]}(S_2) = 2$, and the inequality is strict taking $\mathcal{D}_1 = \mathcal{P}_A$, $\mathcal{D}_2 = \mathcal{GP}_A$

\end{ej}

The following corollary is also obtained from Theorem \ref{comparar en gral}. 

\begin{coro} \label{comparar}

Let $A$ be an Artin algebra. If $M$ is a $A$-module, then there is a Gorenstein projective module $G$ such that $\Phi_{[\mathcal{GP}_A]}(M) = \phi(M\oplus G)$.

\begin{proof}

Since $\fidim(\mathcal{GP}_A) = 0$ and $G \in \mathcal{GP}_A$ is a subcategory with preimages, then the thesis follows by Theorem \ref{comparar en gral}.

\end{proof}
\end{coro}

%
%
%
%
%
%
%

The example below shows that the equality in Corollary \ref{igualdad} is not valid in general.

\begin{ej} \label{ej fi-infinita}
Let $A$ be the algebra from Example \ref{ej sizigia-infinita}. Recall that $\mathcal{D}_0 = \add \{M_{2^{-n}}, N_{-2^{-n}} \forall n \in \mathbb{N}\}$ and consider $\mathcal{D}_{-\infty} = \add\{ \oplus_{ n \in \mathbb{Z}} ( M_{2^n} \oplus N_{2^{n}}) \}$. Then $\Phi_{[\mathcal{GP}_A]} (\mathcal{D}_{-\infty}) = 0$ and $\Phi_{[\mathcal{D}_0]} (\mathcal{D}_{-\infty}) = \infty$, however $\mathcal{D}_0 \subset \mathcal{GP}_A = \mod A$ since $A$ is a selfinjective algebra.  

\end{ej}

The previous generalized Igusa-Todorov functions allow us to define the next class of algebras, a generalization of Igusa-Todorov algebras. 

\begin{defi}

An $(n, V, \mathcal{D})$-Lat-Igusa-Todorov algebra ($(n, V, \mathcal{D})$-LIT-algebra, for short) is an Artin algebra $A$ satisfying the following conditions: n is a non-negative integer, $\mathcal{D} \subset \mod A$ is a $0$-Igusa-Todorov subcategory, and
$V \in \mod A$ such that each $M \in \mod A$ admits a short exact sequence:
$$\xymatrix{0 \ar[r]& V_1 \oplus D_1 \ar[r] & V_0 \oplus D_0 \ar[r] & \Omega^n(M)\ar[r] & 0}$$
such that $V_0, V_1 \in \add V$ and $D_0, D_1 \in \mathcal{D}$.

\end{defi}

\begin{obs} Let $A$ be an Artin algebra.

\begin{itemize}

\item If $A$ is an $n$-IT-algebra, then it is also an $(n, V, \mathcal{D})$-LIT-algebra.
Indeed, just take the same module $V$, appearing in the definition of $n$-IT-algebra, and $\mathcal{D} = \{0\}$.

\item If $A$ is a $n$-Gorenstein algebra, then $A$ is an $(n, V, \mathcal{D})$-LIT-algebra.
Indeed, it is well known that if $M \in \mod A$, $\Omega^n(M) \in \mathcal{G} P_A$. So by taking $\mathcal{D}= \mathcal{G} P_A$ and $V = 0$, we get that $A$ is an $(n, V, \mathcal{D})$-LIT-algebra.

\end{itemize}

\end{obs}

As we can see, the class of the $(n, V, \mathcal{D})$-LIT-algebras is strictly larger than the class of
the $n$-Igusa-Todorov algebras. The following theorem is a generalization of Theorem \ref{IT algebra}.

\begin{teo}(Theorem 5.4 of \cite{BLMV}) Let $A$ be an $(n, V, \mathcal{D})$-LIT-algebra. Then
$$\fin (A) \leq \Psi_{[\mathcal{D}]} (V) + n + 1 < \infty.$$
\end{teo}

Unfortunately, not all Artin algebras are $(n, V, \mathcal{D})$-LIT, as shown in the following example.

\begin{ej}
Let $B =\frac{\Bbbk Q}{I_B}$ be a finite dimensional $\Bbbk$-algebra and $C = \frac{\Bbbk Q'}{J^2}$, where $Q$ and $Q'$ are the following quivers
$$Q = \xymatrix{ & 0 \ar@(ul, dl)_{\gamma_1} \ar@(ur, dr)^{\gamma_2} \ar@(ru, lu)_{\gamma_3} },\ Q'=\xymatrix{ 1 \ar@(ul, dl)_{\bar{\beta}_1} \ar@/^2mm/[r]^{\beta_1} & 2 \ar@(ur, dr)^{\bar{\beta}_2} \ar@/^2mm/[l]^{\beta_2} }, $$ where the ideal $I_B = \langle \gamma_i \gamma_i,\ \gamma_i \gamma_j + \gamma_j\gamma_i \ \forall \ i,j \in \{1,2,3\}$.
Consider $A = \frac{\Bbbk \Gamma}{I_A}$, with

\begin{itemize}

\item $\Gamma_0 = Q_0 \cup Q'_0$,

\item $\Gamma_1 = Q_1 \cup Q'_1 \cup \{ \alpha_i : i \rightarrow 1 \text{ } \forall i \in Q_0\}$ and

\item $I_A = \langle I_B, J^2_{C} , \{ \lambda \alpha_i, \alpha_i\lambda \text{ } \forall \lambda \text{ such that }\length(\lambda)\geq 1 \} \rangle$. 

\end{itemize}

Since $B$ is not an Igusa-Todorov algebra (see 4.2.10 of \cite{T} and Corollary 4.4 of \cite{Rou}), then $A$ is not a {$(n, V, \mathcal{D})$-\rm{LIT}} algebra.

\end{ej}

\subsection{Relative Igusa-Todorov functions}

\begin{defi}

Let $\mathcal{A}$ be an additive category. A  kernel-cokernel pair ${(i, p)}$ in $A$ is a
pair of composable morphisms $A' \stackrel{i}{\rightarrow} A \stackrel{p}{\rightarrow} A''$ such that $i$ is a kernel of $p$ and $p$ is a cokernel of $i$. If a class $\mathcal{E}$ of kernel-cokernel pairs
on $\mathcal{A}$ is fixed,  an admissible monic is a morphism $i$ for which there exists a morphism
$p$ such that $(i, p) \in \mathcal{E}$.  Admissible epics are defined dually.

An exact structure on $\mathcal{A}$ is a class $\mathcal{E}$ of kernel-cokernel pairs which is closed under
isomorphisms and satisfies the following axioms:

\begin{description}

\item[$E0$] For all objects $A \in \mathcal{A}$, the identity morphism $1_A$ is an admissible monic.

\item[${E0}^{op}$] For all objects $A \in \mathcal{A}$, the identity morphism $1_A$ is an admissible epic.

\item[$E1$] The class of admissible monics is closed under composition.

\item[${E1}^{op}$] The class of admissible epics is closed under composition.

\item[$E2$] The push-out of an admissible monic along an arbitrary morphism exists and yields
an admissible monic.

\item[${E2}^{op}$] The pull-back of an admissible epic along an arbitrary morphism exists and yields
an admissible epic.

\end{description}

\end{defi}

\begin{defi}

An  exact category is a pair $(\mathcal{C} , \mathcal{E})$, where $\mathcal{C}$ is an additive category
and $\mathcal{E}$ is an exact structure on $\mathcal{C}$.

\end{defi}

An element $(i, p) \in \mathcal{E}$ in a exact category $(\mathcal{C} , \mathcal{E})$ is usually called short ${\mathcal{E}}$-exact sequence and written as $A'\stackrel{i}{\rightarrowtail} A \stackrel{p}{\twoheadrightarrow} A''$.
An object $P$ in an exact category $(\mathcal{C}, \mathcal{E})$ is ${\mathcal{E}}$-projective if
$Hom_{\mathcal{C}}(P,\ \cdot \ ) : C \rightarrow Ab$ takes any short $\mathcal{E}$-exact sequence to a short exact sequence of
abelian groups. The class of all $\mathcal{E}$-projective objects is denoted by $\mathcal{P}(\mathcal{E})$. It is said that $\mathcal{C}$ has enough ${\mathcal{E}}$-projectives if for any $C \in \mathcal{C}$ there is an admissible epic $P \rightarrow C$ with $P \in \mathcal{P}(\mathcal{E})$. Dually, we have the ${\mathcal{E}}$-injectives, the class $\mathcal{I}(\mathcal{E})$ of
the $\mathcal{E}$-injective objects and the the notion of enough ${\mathcal{E}}$-injectives for $\mathcal{C}$.

\begin{obs}

For an exact category $(\mathcal{C} , \mathcal{E})$, $\add\mathcal{P}(\mathcal{E})= \mathcal{P}(\mathcal{E})$ and  $\add\mathcal{I}(\mathcal{E})= \mathcal{I}(\mathcal{E})$.

\end{obs}

\begin{defi}

An exact IT-context is an exact skeletally small Krull-Schmidt category $(\mathcal{C},\mathcal{E})$, with enough $\mathcal{E}$-projectives. Moreover, if $\mathcal{C}$ is abelian and $\mathcal{E}$ is the class of all exact sequences, we say that $\mathcal{C}$ is an  abelian IT-context. An exact (abelian) IT-context is said to be  strong if $\mathcal{C}$ satisfies the Fitting's property.
\end{defi}

A chain complex $\xymatrix{M^{\bullet} : \ldots \ar[r] & M_{n+1} \ar[r] & M_n \ar[r] & M_{n-1} \ar[r] & \cdots}$ in $\mathcal{C}$, is said to be ${\mathcal{E}}$-acyclic if each differential $d_n$ factors as $M_n \twoheadrightarrow Z_{n-1} (M^{\bullet} ) \rightarrowtail M_{n-1}$ in such a way that each sequence $Z_n (M^{\bullet}) \rightarrowtail M_n \twoheadrightarrow Z_{n-1} (M^{\bullet})$ is $\mathcal{E}$-exact.

\begin{defi}

An $\mathcal{E}$-projective resolution of an object $C \in \mathcal{C}$ is an $\mathcal{E}$-acyclic complex
$$\xymatrix{ \cdots \ar[r] & P_n \ar[r] & P_{n-1} \ar[r] & \cdots \ar[r] & P_1 \ar[r] & P_0 \twoheadrightarrow C},$$

where $P_i \in \mathcal{P} (\mathcal{E})$ for any $i \in \mathbb{N}$. We write this resolution as $P^{\bullet}(C) \twoheadrightarrow C$.

\end{defi}

Dually, we define an $\mathcal{E}$-injective resolution of an object $C \in \mathcal{C}$.

\begin{defi}

Let $(\mathcal{C},\mathcal{E})$ be an exact category with enough $\mathcal{E}$-projectives. In this case, for any $C \in \mathcal{C}$, there is an $\mathcal{E}$-exact sequence $K\rightarrowtail P\twoheadrightarrow C$, with $P \in \mathcal{P} (\mathcal{E})$. The object $K$ is said to be an $\mathcal{E}$-syzygy of $C$, and is written as $\Omega_{\mathcal{E},P}(C)$.

\end{defi}

Dually we define an $\mathcal{E}$-cosyzygy of $C$ of an object $C \in \mathcal{C}$ and we denote it by $\Omega^{-1}_{\mathcal{E},I} (C)$.

\begin{obs}(Proposition 1.15 of \cite{LM2017}) Let $(\mathcal{C},\mathcal{E})$ be any exact category with enough $\mathcal{E}$-projectives. Then, all $\mathcal{E}$-(co)syzygies for an object $C \in \mathcal{C}$ are isomorphic.

\end{obs}

\begin{defi}
For a given exact category $(\mathcal{C},\mathcal{E})$ with enough $\mathcal{E}$-projectives, we introduce
The $\mathcal{E}$-projective dimension for any $C \in \mathcal{C}$ ($\pd_\mathcal{E} (C)$) as the minimal non-negative integer $n$ such that
there is an $\mathcal{E}$-projective resolution
$$\xymatrix{ P_n \rightarrowtail  P_{n-1} \ar[r] & \cdots \ar[r] & P_1 \ar[r] & P_0 \twoheadrightarrow C}.$$
\end{defi}

Dually, we define the $\mathcal{E}$-injective dimension for any $C \in \mathcal{C}$, and we denote it by $\id_\mathcal{E} (C)$.

For each class $\mathcal{Y}$ of objects in $\mathcal{C}$, we set
$$\pd_{\mathcal{E}} ( \mathcal{Y} ) := \sup\{\pd_{\mathcal{E}} (Y ) : Y \in \mathcal{Y} \}.$$
$$\mathcal{P}_{\mathcal{E}}^{ < \infty}( Y ) = \{Y \in \mathcal{Y} : \pd_{\mathcal{E}} (Y ) < \infty\}.$$

\begin{defi} For each class $\mathcal{Y}$ of objects in $\mathcal{C}$, we define the finitistic $\mathcal{E}$-projective dimension of $\mathcal{Y}$ as follows
$$\fpd_{\mathcal{E}} ( \mathcal{Y} ) = \pd_{\mathcal{E}} ( \mathcal{P}_{\mathcal{E}}^{ < \infty}( Y )).$$
\end{defi}

Similarly, we have the class $\mathcal{I}_{\mathcal{E}}^{< \infty} ( \mathcal{Y} )$ and the finitistic $\mathcal{E}$-injective dimension $\fid_{\mathcal{E}} ( \mathcal{Y} )$.

For any $M, N \in \mathcal{C}$ , we denote by $\Fact_{\mathcal{P} (\mathcal{E})} (M, N )$ the set of all morphisms $f : M \rightarrow N$
factoring throughout an object in $\mathcal{P} ( \mathcal{E} )$. Let $N \in \mathcal{C} $ and take any admissible epic $\mu : P \twoheadrightarrow N$ with $P \in \mathcal{P} ( \mathcal{E} )$. Note that, for any $M \in \mathcal{C}$ we have that

$$\Fact_{\mathcal{P} ( \mathcal{E} )} (M, N ) = \I (\hom_{\mathcal{C}} (M, \mu)).$$

Therefore the class of morphisms $\Fact_{\mathcal{P} ( \mathcal{E} )}$ is an ideal in $\mathcal{C}$ , and so, we get the well known
stable category $\underline{\mathcal{C}} := \mathcal{C}/ \Fact_{\mathcal{P} ( \mathcal{E} )}$ modulo projectives. Recall that $\underline{\mathcal{C}}$ and $\mathcal{C}$ have the same objects and $\hom_{\underline{\mathcal{C}}}(M, N ) = \hom_{\mathcal{C}} (M, N )/\Fact_{\mathcal{P} ( \mathcal{E} )} (M, N )$. Dually, we define $\overline{C}$.

\begin{defi}(\cite{LM2017})
Let $(\mathcal{C},\mathcal{E})$ be an exact IT-context. We denote by ${K_\mathcal{E}(\mathcal{C})}$ the quotient of the
free abelian group generated by the set of iso-classes $\{[M] : M \in \mathcal{C} \}$ modulo the relations:
\begin{itemize}

\item $[N] - [S] - [T]$ if $N \cong S \oplus T$ , and

\item $[P]$ if $P \in \mathcal{P}(\mathcal{E}).$

\end{itemize}

\end{defi}

Dually, we define ${K^\mathcal{E}(\mathcal{C})}$ modulo the relations  $[N] - [S] - [T]$ if $N \cong S \oplus T$, and $[I]$ if $I \in \mathcal{I}(\mathcal{E})$.

\begin{obs}
We have that ${K_\mathcal{E}(\mathcal{C})}$ is the free abelian group generated by the
objects in the set $\ind (\mathcal{C})\setminus  \mathcal{P}(\mathcal{E})$. Furthermore, the additive functor
$ \Omega_{\mathcal{E}}: \underline{\mathcal{C}}\rightarrow \underline{\mathcal{C}}$ gives rise to a morphism $\Omega_{\mathcal{E}}: {K_\mathcal{E}(\mathcal{C})} \rightarrow {K_\mathcal{E}(\mathcal{C})}$ of abelian groups, given by
$ \Omega_{\mathcal{E}}([M]) := [\Omega_{\mathcal{E}}(M)]$ for any $M \in \mathcal{C}$.
\end{obs}

Dually, we define $\Omega^{-1}_{\mathcal{E}}: {K^\mathcal{E}(\mathcal{C})} \rightarrow {K^\mathcal{E}(\mathcal{C})}$. 

\begin{defi}(\cite{LM2017})\label{ITGeneralizacion}
Let $(\mathcal{C},\mathcal{E})$ be an exact IT-context. The $\mathcal{E}$-IT functions $\Phi_{\mathcal{E}}$, $\Psi_{\mathcal{E}} : \Obj( \mathcal{C}) \rightarrow\mathbb{N}$ are defined, by using Fitting's Lemma (Lemma \ref{Fitting}), as follows:

\begin{itemize}

\item $\Phi_{\mathcal{E}}(M) = \eta_{\Omega_{\mathcal{E}}} (M)$, and

\item $\Psi_{\mathcal{E}}(M) = \Phi_{\mathcal{E}}(M)+\fpd_{\mathcal{E}} (\add(\Omega_{\mathcal{E}}^{ \Phi_{\mathcal{E}}(M)}(M))).$

\end{itemize}

And its dual versions $\Phi^{\mathcal{E}}$, $\Psi^{\mathcal{E}} : \Obj( \mathcal{C}) \rightarrow\mathbb{N}$ as follows:

\begin{itemize}

\item $\Phi^{\mathcal{E}}(M) = \eta_{\Omega^{-1}_{\mathcal{E}}} (M)$, and

\item  $\Psi^{\mathcal{E}}(M) = \Phi^{\mathcal{E}}(M)+\fid_{\mathcal{E}}(\add(\Omega_{\mathcal{E}}^{-\Phi^{\mathcal{E}}(M)}(M))).$

\end{itemize}

\end{defi}

From Definition \ref{ITGeneralizacion}, we can obtain the classic definitions \ref{monomorfismo} and \ref{ITpsi} and \ref{ITcoalgebras}.

\begin{ej} Special cases 

\begin{itemize} 

\item If $\mathcal{C} = \mod A$ for an Artin algebra $A$, then

\begin{itemize}

\item $\Phi_{\mathcal{P}(A)} (M) = \phi_A(M)$ and $\Psi_{\mathcal{P}(A)} (M) = \psi_A(M)$,

\item $\Phi^{\mathcal{I}(A)} (M) = \phi_{A^{op}}(D(M))$ and $\Psi^{\mathcal{I}(A)} (M) = \psi_{A^{op}}(D(M))$

\end{itemize}

\item If $\mathcal{C} = \mathcal{M}_C$ for a right semiperfect coalgebra $C$, then
$$\Phi^{\mathcal{I}(A)} (M) = \phi_{C}(M) \text{ and } \Psi^{\mathcal{I}(A)} (M) = \psi_{C}(M)$$

\end{itemize}

\end{ej}

\begin{obs} (Propositions 2.10, 2.11, 2.12 of \cite{LM2017}) The following statements hold true for an exact IT-context $(\mathcal{C},\mathcal{E})$.

\begin{enumerate}

\item If $\pd_{\mathcal{E}} (M) < \infty$ then $\Phi_{\mathcal{E}} (M) =  \Psi_{\mathcal{E}} (M) = \pd_{\mathcal{E}} (M)$.

\item If $\pd_{\mathcal{E}} (M) < \infty$ and $M\in \ind (\mathcal{C})$ then $\Phi_{\mathcal{E}} (M) =  \Psi_{\mathcal{E}} (M) = 0$. 

\item If $\add(M) = \add(N )$ then $ \Phi_{\mathcal{E}} (M) =  \Phi_{\mathcal{E}} (N)$ and $ \Psi_{\mathcal{E}} (M) =  \Psi_{\mathcal{E}} (N)$.

\item  $ \Phi_{\mathcal{E}} (M) =  \Phi_{\mathcal{E}} (M \oplus P)$ and $ \Psi_{\mathcal{E}} (M) =  \Psi_{\mathcal{E}} (M \oplus P)$ , for all $M \in \mathcal{C}$ and $P \in \mathcal{P}(\mathcal{E})$.

\item $\Phi_{\mathcal{E}} (M) \leq \Phi_{\mathcal{E}} (M \oplus N )$ and $\Psi_{\mathcal{E}} (M) \leq \Psi_{\mathcal{E}} (M \oplus N )$ for all $M,N \in \mathcal{C}$ .

\item $ \Phi_{\mathcal{E}} (M) \leq  \Phi_{\mathcal{E}} (\Omega_{\mathcal{E}}(M))+1$  and $\Psi_{\mathcal{E}} (M) \leq  \Psi_{\mathcal{E}} (\Omega_{\mathcal{E}}(M))+1$ for all $M \in \mathcal{C}$.

\item Let $Z | \Omega^n_{\mathcal{E}} (M)$, where $n \leq \Phi_{\mathcal{E}} (M)$ and $\pd_{\mathcal{E}} (Z) < \infty$. Then
$\pd_{\mathcal{E}} (Z) + n \leq \Psi_{\mathcal{E}} (M)$.

\end{enumerate}

\end{obs}

There are dual versions for all items in the above remark.

The following result is a generalization of Theorem \ref{cotadp}.

\begin{teo}(Theorem 2.13 of \cite{LM2017}) \label{LMmain} Let $( \mathcal{C}, \mathcal{E})$ be a strong exact IT-context. Then, for any $\mathcal{E}$-exact sequence $A\rightarrowtail B\twoheadrightarrow C$ such that $\pd_{\mathcal{E}} (C) < \infty$, we have that

$$\pd_{\mathcal{E}} (C) \leq \Psi_{\mathcal{E}} (A \oplus B) + 1$$

\end{teo}

The following result is a generalization of Corollary \ref{corocotadp}.

\begin{coro} (\cite{LM2017}) \label{LMcoro} Let $( \mathcal{C}, \mathcal{E})$ be a strong exact IT-context. Then, for any $\mathcal{E}$-exact sequence $A\rightarrowtail B\twoheadrightarrow C$, the following statements hold.

\begin{enumerate}

\item If $\pd_{\mathcal{E}} (B) < \infty $, then $\pd_{\mathcal{E}} (B) \leq 1 + \Psi_{\mathcal{E}} (A \oplus \Omega_{\mathcal{E}} (C))$.

\item If $\pd_{\mathcal{E}} (A) < \infty $, then $\pd_{\mathcal{E}} (A) \leq 1 + \Psi_{\mathcal{E}} ( \Omega_{\mathcal{E}} ( B \oplus C))$.

\end{enumerate}

\end{coro}

\begin{defi}(\cite{LM2017}) Let $(\mathcal{C},\mathcal{E})$ be an exact IT-context. If $\mathcal{Y}$ is a class of objects in a $\mathcal{C}$, we define

\begin{itemize}

\item the $\mathcal{E}$-relative $\Phi$-dimension of $\mathcal{Y}$ as $\Phi_{\mathcal{E}} \dim ( \mathcal{Y}) = \sup\{ \Phi_{\mathcal{E}} (Y) : Y \in \mathcal{Y} \},$ and

\item the $\mathcal{E}$-relative $\Psi$-dimension of $\mathcal{Y}$ as $\Psi_{\mathcal{E}} \dim ( \mathcal{Y}) = \sup\{ \Psi_{\mathcal{E}} (Y) : Y \in \mathcal{Y} \}.$

\end{itemize}

\end{defi}

Dually we define $\Phi^{\mathcal{E}} \dim ( \mathcal{Y})$ and $\Psi^{\mathcal{E}} \dim ( \mathcal{Y})$.

\begin{obs} For any class $\mathcal{Y}$ of objects, in an exact IT-context $(\mathcal{C},\mathcal{E})$, the following
inequalities hold

\begin{itemize}

\item $\fpd_{\mathcal{E}} ( \mathcal{Y} ) \leq \Phi_{\mathcal{E}} \dim ( \mathcal{Y})  \leq  \Psi_{\mathcal{E}} \dim ( \mathcal{Y}) \leq \pd_{\mathcal{E}} ( \mathcal{Y} )$.

\item $\Psi_{\mathcal{E}} \dim ( \mathcal{Y}) \leq \Phi_{\mathcal{E}} \dim ( \mathcal{Y}) + \fpd_{\mathcal{E}} ( \mathcal{Y} ) $.

\item $\Psi_{\mathcal{E}} \dim ( \mathcal{Y}) \leq 2\Phi_{\mathcal{E}} \dim ( \mathcal{Y})$.

\end{itemize}

\end{obs}

The theorem below generalizes Theorem \ref{selfinjective} and Theorem \ref{qcF}.

\begin{teo} (Theorem 3.4 of \cite{LM2017}) Let $\mathcal{C}$ be a skeletally small abelian Krull-Schmidt category, with enough
projectives and injectives. Then $\mathcal{C}$ is a Frobenius category if and only if 

$$\Phi_{\mathcal{P}(\mathcal{C})} \dim (\mathcal{C}) = \Phi^{\mathcal{I}(\mathcal{C})}\dim (\mathcal{C}) = 0$$

\end{teo}

The following example shows that the symmetric condition is necessary.

\begin{ej}\label{contraej} Let $\Bbbk$ be a field and $A_0^{\infty}$ be following infinite quiver

$$A_0^{\infty} = \xymatrix{ 0 \ar[r]^{\alpha_0} & 1 \ar[r]^{\alpha_1} & 2\ar[r]^{\alpha_2} & 3 \ar[r]^{\alpha_3} & \cdots }$$

Consider $\mathcal{C}$ the full subcategory of representations of $\Rep (\Bbbk A_0^{\infty})$ formed by all representation $(M_i, T_{\alpha_i})$ such that $T_{\alpha_{i+1}}\circ T_{\alpha_i} = 0$. 

It is easy to that $P_i = I_{i+1}$ for $i \in \mathbb{N}$ but $I_0$ is not a projective module, thus $\mathcal{C}$ is not a Frobenius category. However $\Phi_{\mathcal{P}(\mathcal{E})} \dim (\mathcal{C}) = 0$ (and $\Phi^{\mathcal{I}(\mathcal{E})} \dim (\mathcal{C}) = \infty$).  

\end{ej}

\end{document}